\documentclass[hidelinks,12pt,a4paper]{article}
\usepackage[utf8x]{inputenc}
\usepackage[top=30mm, bottom=30mm, inner=20mm, outer=20mm, headsep=10mm, footskip=12mm]{geometry}
\usepackage{amssymb}
\usepackage{amsmath,bm}
\usepackage{mathtools}
\usepackage{amsthm}
\usepackage[english]{babel}
\usepackage{lmodern}
\usepackage{microtype}
\usepackage[mathscr]{euscript}
\usepackage[colorlinks=true]{hyperref}
\usepackage[usenames,dvipsnames]{xcolor}
\usepackage{todonotes}
\usepackage{appendix}
\usepackage{tikz}
\usepackage{bm}\usepackage{cancel}
\usepackage{upgreek}
\usepackage[normalem]{ulem}

\numberwithin{equation}{section}

\theoremstyle{plain}
\newtheorem{definition}{Definition}[section]

\newtheorem{theorem}[definition]{Theorem}
\newtheorem{proposition}[definition]{Proposition}
\newtheorem{lemma}[definition]{Lemma}
\newtheorem{corollary}[definition]{Corollary}

\theoremstyle{definition}
\newtheorem{remark}[definition]{Remark}
\newtheorem{example}[definition]{Example}
\newtheorem*{ack}{Acknowledgements}

\makeatletter
\let\save@mathaccent\mathaccent
\newcommand*\if@single[3]{%
  \setbox0\hbox{${\mathaccent"0362{#1}}^H$}%
  \setbox2\hbox{${\mathaccent"0362{\kern0pt#1}}^H$}%
  \ifdim\ht0=\ht2 #3\else #2\fi
  }
\newcommand*\rel@kern[1]{\kern#1\dimexpr\macc@kerna}
\newcommand*\widebar[1]{\@ifnextchar^{{\wide@bar{#1}{0}}}{\wide@bar{#1}{1}}}
\newcommand*\wide@bar[2]{\if@single{#1}{\wide@bar@{#1}{#2}{1}}{\wide@bar@{#1}{#2}{2}}}
\newcommand*\wide@bar@[3]{%
  \begingroup
  \def\mathaccent##1##2{%
    \let\mathaccent\save@mathaccent
    \if#32 \let\macc@nucleus\first@char \fi
    \setbox\z@\hbox{$\macc@style{\macc@nucleus}_{}$}%
    \setbox\tw@\hbox{$\macc@style{\macc@nucleus}{}_{}$}%
    \dimen@\wd\tw@
    \advance\dimen@-\wd\z@
    \divide\dimen@ 3
    \@tempdima\wd\tw@
    \advance\@tempdima-\scriptspace
    \divide\@tempdima 10
    \advance\dimen@-\@tempdima
    \ifdim\dimen@>\z@ \dimen@0pt\fi
    \rel@kern{0.6}\kern-\dimen@
    \if#31
      \overline{\rel@kern{-0.6}\kern\dimen@\macc@nucleus\rel@kern{0.4}\kern\dimen@}%
      \advance\dimen@0.4\dimexpr\macc@kerna
      \let\final@kern#2%
      \ifdim\dimen@<\z@ \let\final@kern1\fi
      \if\final@kern1 \kern-\dimen@\fi
    \else
      \overline{\rel@kern{-0.6}\kern\dimen@#1}%
    \fi
  }%
  \macc@depth\@ne
  \let\math@bgroup\@empty \let\math@egroup\macc@set@skewchar
  \mathsurround\z@ \frozen@everymath{\mathgroup\macc@group\relax}%
  \macc@set@skewchar\relax
  \let\mathaccentV\macc@nested@a
  \if#31
    \macc@nested@a\relax111{#1}%
  \else
    \def\gobble@till@marker##1\endmarker{}%
    \futurelet\first@char\gobble@till@marker#1\endmarker
    \ifcat\noexpand\first@char A\else
      \def\first@char{}%
    \fi
    \macc@nested@a\relax111{\first@char}%
  \fi
  \endgroup
}
\makeatother

\newcommand{\R}{\mathbf R}

\newcommand{\N}{\mathbf N}

\newcommand{\Vol}{\mathrm{Vol}}

\renewcommand{\d}{\,\mathrm{d}}

\usepackage[xcolor]{changebar}
\cbcolor{blue}
\newcommand{\Ric}{\mathrm{Ric}}

\usepackage[inline]{enumitem}
\newcommand{\enumlabelformat}{\roman}

\newlength{\thelabelsep}
\newcounter{inlineenum}
\renewcommand{\theinlineenum}{\enumlabelformat{inlineenum}}
\let\epsilon\varepsilon
\let\phi\varphi

\newcommand{\diam}{\mathrm{diam}}

\newcommand{\Prob}{\mathscr{P}}

\DeclareMathOperator{\supp}{spt}

\newcommand{\eval}{\mathsf{e}}

\newcommand{\TGeo}{\mathrm{TGeo}}

\newcommand{\e}{\eval}
\newcommand{\nchi}{{\raise.3ex\hbox{$\chi$}}}

\usepackage[runin]{abstract}

\newcommand{\LpLS}{Lo\-rentz\-ian pre-length space }

\newcommand{\LpLSs}{Lo\-rentz\-ian pre-length spaces }

\DeclareMathOperator{\Sec}{Sec}

\renewcommand{\d}{\,\mathrm{d}}

\makeatletter
\let\@fnsymbol\@arabic
\makeatother

\allowdisplaybreaks

\title{Submetries in non-positive signature and applications}

\author{Matteo Zanardini \footnotemark[1]}

\date{\today}

\begin{document}

\maketitle
\footnotetext[1]{SISSA, 34136, Trieste, Italy, mzanardi@sissa.it} 
\begin{abstract}
    We consider a generalization of the notion of submetry  tailored to the setting of spacetimes. We show that, under suitable completeness assumptions, Lorentzian submetries between smooth spacetimes correspond to locally $C^{1,1}$ Semi-Riemannian submersions. Moreover, we establish some applications regarding timelike curvature bounds, isometric group actions, acausal foliations and Lorentz-Wasserstein geometry.
    %\bigskip

%\noindent
%\emph{Keywords:} Lorentzian geometry, Lorentzian submetries, optimal transport, curvature bounds
%\medskip

%\noindent
%\emph{MSC2020:}
%53C23, %Global geometric and topological methods (à la Gromov); differential geometric analysis on metric spaces
%51K10, %Synthetic differential geometry
%53C50, %Lorentz manifolds, manifolds with indefinite metrics 
%53B30 %Lorentz metrics, indefinite metrics
%in the class of Lorentzian pre-length spaces.
\end{abstract}

\tableofcontents

\section{Introduction}

In some areas of Differential Geometry, it is of particular interest the characterization of objects given in the smooth category with a suitable synthetic formulation that typically does not require any smoothness assumption. An example of such a result in Riemannian Geometry is the Theorem of Myers-Steenrod \cite{MyersSteenrod} that asserts that a function between Riemannian manifolds is a smooth isometry if and only if it is surjective and distance preserving. This kind of equivalence has also been established in the Lorentzian setting by Hawking-King-McCarthy \cite{HKM} in a foundational result asserting that a map between two strongly causal spacetimes is a smooth isometry if and only if it is bijective and preserves the respective time separation functions. This pattern of results, among other things, contributes to the development of the theory of non-smooth Differential Geometry.

In the Riemannian setting, Berestovskii defined in \cite{Berestovskii} the non-smooth analogue of a Riemannian submersion: a map $F$ between metric spaces is called a submetry if it sends closed balls of radius $R$ around a point to closed balls of radius $R$ around the image point. In \cite{BerGui}, the following compatibility result was established: a map between complete Riemannian manifolds is a submetry if and only if it is a locally $C^{1,1}$ Riemannian submersion. On the other hand, the regularity theory of submetries between Riemannian manifolds is richer and has had very interesting developments in recent times, see \cite{lytchak2024smoothness}, \cite{lytchak2024some} and \cite{KapoLytchak}. Moreover, the notion of submetries is related to rigidity results such as \cite{BerGui}, \cite{GroveGromoll}, as well as the study of curvature bounds either in the smooth setting or with a suitable synthetic formulation, e.g. \cite{BerGui}, \cite{QuotientMondino}. 

In this work we propose a generalization of the notion of submetries tailored to the setting of Lorentzian Geometry. Since this definition does not require any smoothness, we will give it, and study some properties, in the non-smooth setting of Lorentzian pre-length spaces, introduced by Kunzinger-Sämann in \cite{KunSam}.

In section \ref{sec: submetries} we give the definition of Lorentzian submetries, see Definition \ref{def: lorentzian submeetry}, and we study some properties of such maps. For instance, we establish a Lifting construction, and obtain some geometric properties. Under suitable completeness assumptions on the source and target space, we show that Lorentzian submetries between smooth spacetimes correspond to locally $C^{1,1}$ Lorentzian submersions; the main theorem of this work is summarized in the following
\begin{theorem} \label{thm: main}
    Let $(M,g), (N,h)$ be globally hyperbolic and causal geodesically complete spacetimes and $F: M \rightarrow N$ be a map. Then $F$ is a Lorentzian submetry if and only if $F$ is a surjective Lorentzian submersion of class $C^{1,1}_{loc}$.
\end{theorem}
Regarding the equivalence contained in Theorem \ref{thm: main}, the assumption of causal geodesic completeness is needed for only one of the two implications, indeed see Theorem \ref{thm: regularity} and Remark \ref{rem: countr}. Moreover, building on top of the Riemannian case, we show that the regularity of the map cannot be improved to $C^2$ and, in Example \ref{example desitter}, we explicitly build a Lorentzian submetry from De Sitter spacetime which is not $C^2$ by exploiting a cut-and-paste construction with Busemann functions. 

In the last section we provide some applications of Lorentzian submetries between spacetimes with either timelike sectional curvature bounds or satisfying the Strong Energy Condition \eqref{SEC}. Moreover, we establish a correspondence between submetries of a globally hyperbolic spacetime with (possibly singular) acausal and equidistant foliations. Finally we study isometric group actions on globally hyperbolic spacetimes and we show some properties of Lorentz-Wasserstein spaces that permit us to relate Lorentzian submetries and Ricci curvature bounds.

\section{Preliminaries}

\subsection{Spacetime geometry}

In this section we give a short presentation of the geometric setting we will work with, for more details we suggest \cite{Min19, BeemBook, Octet}. We start with smooth spacetimes and then we move to the description of their non-smooth counterparts.

A spacetime $(M,g)$ is a connected and non-compact smooth manifold $M$ endowed with a Lorentzian metric $g$ of signature $(+,-, \dots ,-)$ that is time oriented by a vector field $X$. A vector $v \in T_xM$ is called timelike if $g_x(v,v) >0$, null if $g_x(v,v) =0$ and spacelike if $g_x(v,v) <0$. Morever we say that $v$ is causal provided it is either timelike or null. A causal vector $v$ is called future directed provided $g(v,X) \geq 0$. A differentiable curve $\gamma:[0,1] \rightarrow M$ is called causal (resp. timelike or resp. null) if $\dot \gamma_t$ is causal (resp. timelike or resp. null) for every $t \in [0,1]$. The collection of future directed causal vectors is denoted by $F_xM \subset T_xM$ and we define the hyperbolic norm out of $g$ on the tangent space $T_xM$ by
\[
\| v \|_g : = \begin{cases}
    \sqrt{g_x(v,v)} & \text{ if } v \in F_xM \\
    0 & \text{ otherwise. } 
\end{cases}
\]
We refer to \cite[Section A.3]{Octet} and \cite{HBS} for a discussion on the topic. Moreover, we define:
\begin{enumerate}
    \item[i)] the causal order $\leq$: given two points $x_0, x_1 \in M$ we say that $x_0 \leq x_1$ provided there exists a future directed piecewise $C^1$ causal curve $\gamma$ that connects $x_0$ to $x_1$,
    \item[ii)] the chronological order $\ll$: given two points $x_0, x_1 \in M$ we say that $x_0 \ll x_1$ provided there exists a future directed piecewise $C^1$ timelike curve $\gamma$ that connects $x_0$ to $x_1$,
    \item [iii)] the time separation function $\ell$: given two points $x_0, x_1 \in M$ we define
    \[
    \ell (x_0, x_1) = \sup \int_0^1 \| \dot \gamma_t \|_g \, \d t \, ,
    \]
    where the supremum is computed with respect to all future directed piecewise $C^1$ causal curves $\gamma:[0,1] \rightarrow M$ such that $\gamma_0=x_0$ and $\gamma_1 = x_1$. If there are not such curves, we say that $\ell(x_0, x_1) =0$ by convention.
\end{enumerate} 
It is easy to check that if $x_0 \ll x_1$ then $x_0 \leq x_1$ and $\ell(x,y)>0$ if and only if $x_0 \ll x_1$. The causal order is reflexive and transitive, moreover if $x\leq y\leq z$ the time separation $\ell$ satisfies the reversed triangle inequality $\ell(x,z) \geq \ell(x,y) + \ell(y,z)$. The Lorentzian length of a curve will be denoted by $L_g(\gamma)$. A set $U \subset M$ is called causally convex if for every $\gamma :[0,1] \rightarrow M$ causal curve such that $\gamma_0, \gamma_1 \in U$ then $\gamma_t \in U$ for every $t \in [0,1]$. 

In the following definition we collect some causality and topological aspects of spacetimes.
\begin{definition}
    A spacetime $(M,g)$ is called strongly causal if it admits arbitrarily small causally convex neighborhoods at every point. $(M,g)$ is called globally hyperbolic if there are no closed causal curves and for every couple of points $x,z \in M$ we have that $\{y \in M : x \leq y \leq z\} \subset M$ is compact. Finally, a spacetime is timelike (resp. causal) geodesically complete if the exponential map $\exp_x$ is well defined for every timelike (resp. causal) vector $v \in T_xM$.
\end{definition}
Every point on a strongly causal spacetime admits causally convex normal and globally hyperbolic local basis for the topology, see \cite[Theorem 2.7]{Min19}. 

On a globally hyperbolic spacetime the time separation function is finite and continuous, moreover for every $x_0 \leq x_1$ there exists a maximizing causal geodesic $\gamma:[0,1] \rightarrow M$ such that $\gamma_0= x_0$ and $\gamma_1=x_1$. We recall that if $x \leq y \leq z$ and $\ell(x, z) > 0$, the reversed triangle inequality holds strictly unless $y$ lies on a maximizing timelike geodesic joining $x$ to $z$.

We say that a spacetime $(M,g)$ satisfies the Strong Energy Condition \eqref{SEC} provided
\begin{equation} \tag{SEC} \label{SEC}
    \Ric_x (v,v) \geq 0 \qquad \quad \text{ holds for any causal vector } v \in T_xM \, .
\end{equation}
Moreover, a timelike line is a smooth inextensible timelike geodesic $\gamma: \R \rightarrow M$ parametrized by $g$-arclength that is globally maximizing i.e.
\[
\ell_M (\gamma_s, \gamma_t) = t-s \qquad \text{for any $t>s$} \, . 
\]
In some applications we are going to use the Lorentzian Splitting Theorem \ref{thm: splitting}: see \cite{Splitting} and also \cite{ellipticsplitting} for a more recent proof with an elliptic flavor using the $p$-d'Alambertian operator. 

In \cite{KunSam}, Kunzinger-Sämann introduced the novel geometric framework of Lorentzian pre-length spaces that permits to generalize the notion of spacetimes, allowing non-smooth features. We give a short presentation of these spaces in the following definition.

\begin{definition}
    A causal space is a set $X$ endowed with a preorder $\leq$ and a transitive relation $\ll$ contained in $\leq$. A function $\ell : X \times X \rightarrow [0, \infty]$ is a called time separation function if satisfies
    \[
    \ell(x,z) \geq \ell(x,y) + \ell(y,z) \qquad \qquad \text{for any $x\leq y\leq z$}
    \]
    and $\ell(x,y) >0$ if and only if $x\ll y$ and $\ell(x,y) =0$ if $x \not \leq y$. A Lorentzian pre-length space $(X,\mathsf d,\ll,\leq,\ell)$ is a causal space $(X,\ll,\leq)$ endowed with a proper distance $\mathsf d$ (i.e., closed and bounded subsets are compact) and a lower semicontinuous time separation function $\ell$. 
\end{definition}
On a Lorentzian pre-length space there is a nice interplay between the causality encoded by $\leq, \ll,\ell$ and the topology that comes from the distance $\mathsf d$. We suggest \cite{KunSam} for some more properties of these spaces. On the other hand, we recall the alternative approach of metric spacetimes introduced in \cite{Octet}, on which the time separation function is allowed to take value $-\infty$ and it encodes completely the relations $\leq$ and $\ll$. 

On a given causal space endowed with a time separation function $(X,\ell)$ we can define the outer balls in the following way
\begin{equation} \label{eq: outer balls}
    \mathsf{O}^+(x, r) : = \{ y \geq x:\ell(x,y) \geq r \} ,\qquad   \mathsf{O}^-(x, r) : = \{ y \leq x :\ell(y,x) \geq  r \} \, , 
\end{equation}
for any given $x \in X$ and $r \geq 0$. Notice that, if we take $r=0$ in \eqref{eq: outer balls} we recover the standard causal future and causal past of the point $x$, i.e.
\[
\mathsf{O}^+(x, 0) = J^+(x),  \qquad \mathsf{O}^-(x, 0) = J^-(x) \, .
\]
In this manuscript, in order to avoid any confusion, we will denote the outer balls by $\mathsf{O}^+_X(x, r)$ whenever it not clear from the context to what space we refer to. A map $F: (X, \ell_X) \rightarrow (Y, \ell_Y)$ between causal spaces endowed with a time separation function is said to be $1$-steep provided $ \ell_Y(F(x), F(y)) \geq \ell_X(x,y)$ for every $x \leq y$ in $X$. 

We will refer to the set $X^2_{\leq} \subset X \times X$ as the set of $\leq$-related points and a subset $A \subset X$ is called acausal if for $x, y \in A$ we have that $x \leq y$ implies $x=y$. A subset $A \subset X$ is called achronal if $x \not \ll y$ for all $x,y \in A$. Moreover, if $A,B \subset X$ we will denote
\[
\ell(A,B) : = \sup \, \{ \ell(x,y) : x \in A, y \in B\} \, .
\]
A curve $\gamma \in C([0,1];X)$ is called a timelike geodesic on $X$ if it is $\mathsf d$-Lipschitz continuous, obeys $\gamma_s \ll \gamma_t$ for every $s, t \in [0, 1]$ with $s < t$, and maximizes the $\ell$-length. The set of timelike geodesics on $X$ is denoted by $\TGeo([0,1];X)$ and corresponds to the maximizing timelike geodesics, parametrized with constant speed on $[0,1]$ in globally hyperbolic spacetimes.

A subset $A \subset C([0,1];X)$ is called non-branching if for all $\gamma,\sigma \in A$ and $t \in (0,1)$ such that $\gamma_s= \sigma_s$ for all $s \in [0,t]$ or $s \in [t,1]$ then $\gamma=\sigma$. 

\subsection{Lorentzian optimal transport}

In this last subsection of the preliminaries, we collect some definitions and basic material on Lorentzian optimal transport. We suggest \cite{Braun, McCann, Octet, CavMon} for relevant information on the topic. Let $X$ be a Lorentzian pre-length space and, given two probability measures $\mu,\nu\in\mathscr{P}(X)$, we denote by $\Pi_{\leq}(\mu,\nu)$ the set of causal couplings, i.e.\ transport plans which are concentrated on causally related points. 
More precisely, 
\begin{equation*}
    \Pi_{\leq}(\mu,\nu): = \left \{ \pi \in \Prob(X \times X) : ( \mathsf{Pr}_{1})_{\#} \pi = \mu ,  ( \mathsf{Pr}_{2})_{\#} \pi = \nu \text{ and } \pi ( X^2_{\leq})=1 \right \},
\end{equation*}
where $\mathsf{Pr}_1,\mathsf{Pr}_2: X \times X \rightarrow X$ are the natural projections on the first and second component, respectively. Moreover, we say that $\mu \preceq \nu$ provided $\Pi_{\leq} (\mu, \nu) \not = \emptyset$. In an analogous way one can define the set of timelike couplings $\Pi_{\ll}(\mu, \nu)$. 

Given $q \in (0,1)$, the $q$-Lorentz--Wasserstein distance between $\mu,\nu\in\Prob(M)$ is denoted by $\ell_q(\mu,\nu)$ and it is $0$ if $\Pi_{\leq} (\mu, \nu) = \emptyset$, otherwise it is defined by the element in $[0,+ \infty]$ such that
\begin{equation} \label{eq: p- wassestein}
\ell_q(\mu,\nu)^q:= \sup_{\pi\in \Pi_{\leq}(\mu,\nu)}  \int  \ell(x,y )^q \, \pi(\d x,\d y) \, .
\end{equation}
A causal coupling $\pi$ of $\mu$ and $\nu$ is called $\ell_q$-optimal if it attains the supremum in \eqref{eq: p- wassestein}. 

In this setting, it is possible to show that $(\Prob(X), \preceq , \ell_q)$ is a causal space endowed with a time separation and if $\leq$ is a partial order on $X$, then $\preceq$ is a partial order as well, see \cite[Proposition 2.16]{Octet}. In the following definition, we consider particular couples of probability measures, introduced in \cite[Definition 4.1]{McCann}.

\begin{definition}[$q$-separated measures] \label{def: separat}
    Let $q \in (0,1)$. We say that the couple $(\mu,\nu)\in\Prob_c(X)^2$ is $q$-separated provided there exist $\pi\in \Pi_{\leq}(\mu,\nu)$ and the lower semicontinuous functions $u:\supp(\mu)\to \R\cup\{+ \infty\}$ and $v:\supp(\nu)\to \R\cup\{-\infty\}$ such that
    \[
        u(x)+v(y) \geq \tfrac{1}{q} \ell(x,y)^q \qquad \text{ for all } (x,y)\in \supp(\mu)\times \supp(\nu),
    \]
    with  $\supp(\pi)\subseteq S:=\left \{(x,y)\in \supp(\mu)\times \supp(\nu): u(x)+v(y)=\frac 1q \ell(x,y)^q \right \}$ and $S\subseteq \{\ell>0\}$.  
\end{definition}
A probability measure $\boldsymbol{\pi} \in \Prob(C([0,1];X))$ is said to represent the path $[0,1] \ni t \mapsto \mu_t \in \Prob(X)$ provided $(\e_t)_{\#}\boldsymbol{\pi}= \mu_t$ for all $t \in [0,1]$, where $\e_t$ is the evaluation map of at time $t$. \\
Accordingly to \cite{McCann}, we will refer to a $q$-geodesic in $\Prob (X)$ by a weakly continuous curve $[0,1] \ni t \mapsto \mu_t \in \Prob (X)$ satisfying
\[
    \ell_q(\mu_s,\mu_t)= (t-s) \, \ell_q(\mu_0,\mu_1)\in
    (0,\infty) \quad \text{for all } 0\leq s<t\leq 1\, .
\]
Let us endow $X$ with a Radon measure $\mathfrak{m}$ with full support. We will refer to such a couple as a measured Lorentzian pre-length space. 
\begin{definition}
    Let $q \in (0,1)$. A measured Lorentzian pre-length space is said to be $q$-essentially non-branching if every $\boldsymbol{\pi}$ representing a $q$-geodesic between measures in $\Prob^{ac}_c(X, \mathfrak m)$ is concentrated on a non-branching subset.
\end{definition}
Given any $N \in [1, \infty)$ we define the $N$-Renyi entropy $S_N : \Prob (X) \rightarrow [-\infty, 0]$ by
\begin{equation}
    S_N(\mu) = - \int_X \rho(x)^{1-\tfrac{1}{N}} \, \mathfrak{m}(\d x) \, , 
\end{equation}
where $\mu= \rho \mathfrak{m}+ \mu^{\perp}$ is the Lebesgue decomposition with respect to $\mathfrak{m}$.
Given $t \in[0,1]$, $K \in \R$, and $N \in [1,\infty)$ it is possible to define the distortion coefficients at $\theta \geq 0$ via
\begin{equation}
    \sigma^{(t)}_{K,N} (\theta) : = \begin{cases}
        \frac{\sin_{K/N}(t\theta)}{\sin_{K/N}(\theta)} & \text{if } K\theta^2 < N \pi^2 \, , \\
        + \infty & \text{otherwise , }
    \end{cases}  \qquad \tau_{K,N}^{(t)}(\theta) : = t^{1/N} \sigma_{K,N-1}^{(t)} (\theta) ^{1-1/N} \, .
\end{equation}

Finally, in the following definition we introduce the $\mathsf{TCD}_q(K,N)$ condition. We suggest \cite{Braun} for other similar notions of timelike curvature-dimension conditions and, among other things, the proof of the equivalence of them in the $q$-essentially non-branching case. 
\begin{definition} \
    Let $q \in (0,1)$, $K \in \R$, and $N \in [1,\infty)$. We say that the Lorentzian pre-length space $X$ endowed with the fully supported Radon measure $\mathfrak{m}$ satisfies the $\mathsf{TCD}_q(K,N)$ condition if for every  $q$-separated $(\overline{\mu}_0, \overline{\mu}_1) \in \Prob_c^{ac}(X, \mathfrak{m})^2$ there exists $q$-geodesic $(\mu_t)$ connecting $\overline{\mu}_0$ to $\overline{\mu}_1$ and an $\ell_q$-optimal coupling $\pi \in \Pi_{\leq}(\overline{\mu}_0, \overline{\mu}_1)$ such that for every $t \in[0,1]$ and every $N' \geq N$ we have
    \begin{equation} \label{eq: TCD ineq}
        S_N (\mu_t) \leq - \int \tau_{K, N'}^{(1-t)} (\ell(x,y)) \,  \rho_0(x)^{-\tfrac{1}{N'}} + \tau_{K, N'}^{(t)} (\ell(x,y)) \,  \rho_1(y)^{-\tfrac{1}{N'}} \, \pi (\d x,\d y) \, .
    \end{equation}
\end{definition}

We recall that, in a globally hyperbolic spacetime, given $q$-separated measures $(\overline{\mu}_0, \overline{\mu}_1) \in \Prob_c^{ac}(M, \Vol_g)^2$ there exists a unique $q$-geodesic $(\mu_t)$ connecting $\overline{\mu}_0$ to $\overline{\mu}_1$ and a unique $\ell_q$-optimal coupling $\pi \in \Pi_{\leq}(\overline{\mu}_0, \overline{\mu}_1)$, see \cite[Theorem 5.8 \& Corollary 5.9]{McCann}. 

We conclude the preliminary section by recalling a compatibility result with the smooth setting pioneered by McCann in \cite{McCann}, see also \cite{MondinoSuhr}, \cite{Braun} and \cite{BraunOhta} for the equivalence result in the Lorentz-Finsler setting.

\begin{theorem} \label{thm: compatibility}
Let $q \in (0,1)$, $K \in \R$ and $N \geq 1$. Then, a globally hyperbolic spacetime $(M,g)$ with dimension $\dim(M) \leq N$ endowed with the volume measure $\Vol_g$ satisfies the $\mathsf{TCD}_q(K,N)$ condition if and only if $\Ric(v,v) \geq Kg(v,v)$ for any causal vector $v$. 
\end{theorem}

\section{The Lorentzian submetries} \label{sec: submetries}

\subsection{Submersions between spacetimes}

The aim of this section is to study Lorentzian submersions between spacetimes and introduce a characterization of these maps using the time separation functions. We suggest the works by O'Neill \cite{oneillBook, OneillSUB} for an introduction to the more general Semi-Riemannian submersions between Semi-Riemannian manifolds and their associated tensorial objects. In the following definition we recall the notion of Lorentzian submersion.

\begin{definition} [Lorentzian submersion] \label{def: Lor submersion}
    Let $(M,g)$ and $(N,h)$ be spacetimes and $F: M \rightarrow N$ be a differentiable map. We say that $F$ is a Lorentzian submersion if, for every $x \in M$, the map
    \[
    \d F_x \, \big |_{\ker (\d F_x)^{\perp}} \, : \ker (\d F_x)^{\perp} \rightarrow T_{F(x)} N
    \]
    is an isometry between Lorentzian linear spaces that preserves the time orientation.
\end{definition}

\begin{remark}
    Notice, that if $F: M \rightarrow N$ is a Lorentzian submersion between spacetimes in the sense of Definition \ref{def: Lor submersion} then $\ker(\d F_x) \subset T_xM$ is a spacelike subspace since its orthogonal complement is isometric to a Lorentzian tangent space. Moreover, the decomposition 
    \[
    T_xM = \ker (\d F_x)^{\perp}  \oplus \ker (\d F_x) \, ,
    \]
    permits to say that a vector $v \in T_xM$ is horizontal provided $v \in \ker (\d F_x)^{\perp}$ and it is vertical provided $v \in \ker (\d F_x)$. It is clear that every non-zero vertical vector is spacelike. The vertical distribution induced by a smooth Lorentzian submersion is always integrable by the fibers of the map. The same result for the horizontal distribution is not true in general: an obstruction to integrability is given by the O'Neill tensor $A$, see \cite{OneillSUB}. \hfill$\blacksquare$
\end{remark}

The following lemma shows that a $C^1$ Lorentzian submersion is always a $1$-steep map.

\begin{lemma} \label{lem: steep}
    Let $(M,g)$, $(N,h)$ be spacetimes and $F: M \rightarrow N$ be a $C^1$ Lorentzian submersion. Then $F$ is a $1$-steep map.
\end{lemma}

\begin{proof}
Let $x_0 \leq x_1$ in $M$ and $\gamma :[0,1] \rightarrow M$ be a piecewise $C^1$ future directed $g$-causal curve such that $\gamma_0 = x_0$ and $\gamma_1 = x_1$. At each time $t \in [0,1]$ for which $\dot \gamma_t$ exits we consider its decomposition with respect to $T_{\gamma_t}M = \ker(\d F_{\gamma_t}) \oplus \ker(\d F_{\gamma_t})^{\perp}$, say $\dot \gamma_t = (\dot \gamma_t)^{\parallel} + (\dot \gamma_t)^{\perp}$. It is clear that $(\dot \gamma_t)^{\parallel} \in \ker(\d F_{\gamma_t})$ is $g$-spacelike (or zero), and so $\| (\dot \gamma_t)^{\perp} \|_g \geq \| \dot \gamma_t \|_g \geq 0$. Moreover, since $F$ is a Lorenztain submersion, we have that $\d F_{\gamma_t} [\dot \gamma_t ]$ is future directed and 
    \begin{equation} \label{eq: id norm}
     \|\d F_{\gamma_t} [\dot \gamma_t ] \|_h = \| (\dot \gamma_t )^{\perp} \|_g \, .
    \end{equation}
    Moreover, $\sigma_t : = F \circ \gamma_t : [0,1] \rightarrow N$ is a piecewise $C^1$ future directed $h$-causal curve since $\dot \sigma _t = \d F_{\gamma_t} [\dot \gamma_t ]$ exists at each point for which $\dot \gamma_t$ exists and it is manifestly a future directed $h$-causal vector. Hence 
    \begin{equation}
        \ell_N (F(x_0), F(x_1)) \geq \int_0^1 \| \dot \sigma_t \|_h \,  \d t= \int_0^1 \| (\dot \gamma_t )^{\perp} \|_g \, \d t \geq \int_0^1 \| \dot \gamma_t \|_g \, \d t \, .
    \end{equation}
    By arbitrariness of the curve $\gamma$ we get that $F$ is $1$-steep.
\end{proof}

In the following lemma we exploit a first lifting construction using a $C^1$ Lorenztain submersion.

\begin{lemma} \label{lem: C1 lifting}
    Let $(M,g)$ be a strongly causal and timelike geodesically complete spacetime. Let $(N,h)$ be a strongly causal spacetime and let $F:M \rightarrow N$ be a surjective $C^1$ Lorentzian submersion. Let $x \in N$, $y \in F^{-1}(x)$ and $\gamma:[-a,b] \rightarrow N$ be a timelike geodesic parametrized by $h$-arclength such that $\gamma_0=x$. Then there exists a timelike geodesic $\sigma:[-a,b] \rightarrow M$ parametrized by $g$-arclength such that $\sigma_0=y$, $F(\sigma_t)= \gamma_t$ for every $t \in [-a,b]$ and $L_g(\sigma) = L_h(\gamma)$. In particular, $(N,h)$ is timelike geodesically complete.
\end{lemma}

\begin{proof}
Let $Z$ be a smooth timelike geodesic vector field defined on an open set $U \subset N$ for which $\gamma$ is an integral curve. We can lift $Z$ to a continuous horizontal timelike vector field $\tilde{Z}$ defined on $F^{-1}(U) \subset M$ by
\[
F^{-1}(U) \ni z \mapsto \tilde Z_{z} : = \left [ \d F_z \big |_{\ker (\d F_z)^{\perp}} \right ] ^{-1} \circ Z_{F(z)} \, ,
\]
moreover notice that $\| Z_{F(z)} \|_h= \|\tilde Z_z\|_g$. Hence, by Peano's local existence Theorem, there exists a $C^1$ curve $\sigma_t$ such that $\dot \sigma_t = \tilde Z_{\sigma_t}$ and $\sigma_0=y$. Notice that $F\circ \sigma_t$ is a $C^1$ curve satisfying 
\[
\dfrac{\d}{\d s} \bigg |_{s=t} F\circ \sigma_s = Z_{F(\sigma_t)} \, .
\]
Hence $F \circ \sigma_t$ is a timelike geodesic on $N$ with $F(\sigma_0)=x$ and $Z_{x}= \dot \gamma_0$. In particular, by uniqueness, we have that $F\circ \sigma_t= \gamma_t$ and, for $\varepsilon >0$ sufficiently small, by Lemma \ref{lem: steep} we have that
\[
L_g \left (\sigma \big |_{(-\varepsilon, \varepsilon)} \right ) \leq \ell_M (\sigma_{-\varepsilon}, \sigma_{\varepsilon}) \leq \ell_N (\gamma_{-\varepsilon}, \gamma_{\varepsilon}) = L_h \left (\gamma \big |_{(-\varepsilon, \varepsilon)} \right ) = L_g \left (\sigma \big |_{(-\varepsilon, \varepsilon)} \right ) \, .
\]
It is clear, that $\sigma$ is a maximizing timelike geodesic on $M$ parametrized by $g$-arclength. Since $(M,g)$ is timelike geodesically complete it follows that $\sigma$ is defined for every $t \in \R$. The identity $F\circ \sigma_t = \gamma_t$ implies that $(N,h)$ is timelike geodesically complete.
\end{proof}

In the following proposition we show that, under some assumptions, a Lorentzian submersion maps outer balls with some given radius at $x$ in outer balls of the same radius and origin $F(x)$.

\begin{proposition} \label{prop: sumbetry implies submersion}
    Let $(M,g)$ and $(N,h)$ be globally hyperbolic spacetimes. Suppose that $(M,g)$ is causal geodesically complete and let $F: M \rightarrow N$ be a surjective $C^1$ Lorentzian submersion. Then, for any $x \in M$ and $r \geq 0$, we have that
    \[
    F( \mathsf{O}_M^+(x, r) ) = \mathsf{O}_N^+(F(x), r) , \, \qquad \, F( \mathsf{O}_M^-(x, r) ) = \mathsf{O}_N^- (F(x), r) \, .
    \]
\end{proposition}

\begin{proof}
    Since $F$ is $1$-steep by Lemma \ref{lem: steep} we get that 
    \[
    F( \mathsf{O}_M^+(x, r) ) \subseteq \mathsf{O}_N^+(F(x), r) \, \qquad \, F( \mathsf{O}_M^-(x, r) ) \subseteq \mathsf{O}_N^- (F(x), r) \, .
    \]
    We are left to show the other inclusion, we just show the one regarding $\mathsf O^+$ since the other one is completely analogous. Let $z \in \mathsf{O}_N^+(F(x), r)$ and suppose that $\ell_N(F(x),z)= \eta \geq r$. Then, since $(N,h)$ is globally hyperbolic, there exists a maximizing causal geodesic $\gamma:[0,1] \rightarrow N$ such that $\gamma_0= F(x)$, $\gamma_1= z$ and 
    \[
    \int_0^1 \| \dot \gamma_t\|_h \, \d t = \eta \geq 0 \, .
    \]
    If $\eta >0$, by Lemma \ref{lem: C1 lifting} there exists a timelike geodesic $\sigma:[0,1] \rightarrow M$ such that $\sigma_0=x$, $F(\sigma_t)= \gamma_t$ and $\dot \sigma _t \in \ker(\d F_{\sigma_t})^{\perp}$ for every $t \in [0,1]$. In particular,
    \[
    \ell_M (x, \sigma_1) \geq \int_0^1 \| \dot \sigma _t \|_g \, \d t= \int_0^1 \|\dot \gamma_t \|_h \, \d t \geq \eta \, .
    \]
    That is to say $\sigma_1 \in \mathsf{O}^+(x,r)$ and so $F(\sigma_1) = \gamma_1 =z \in F( \mathsf O^{+}(x,r))$ and so, the claim follows. 
    
    If $\eta=0$, we consider $z_n \rightarrow z$ such that $F(x) \ll z_n$ for every $n \in \N$. The same construction as before yields the existence of timelike vectors $v_n \in T_{F(x)}N$ such that 
    \[
    \exp^N_{F(x)}(v_n) = z_n, \qquad v_n \rightarrow v \in T_{F(x)}N, \qquad \exp_{F(x)}(v)=z \, .
    \]
    Moreover, by Lemma \ref{lem: C1 lifting} there exists horizontal future directed timelike vectors $w_n \in T_xM$ such that $\d F_x [w_n]=v_n$. Notice that $w_n \rightarrow w$. In particular, $w$ is a future directed causal vector. Define $x_n : = \exp_x(w_n)$, then $F(x_n)= z_n$ holds as in the construction of Lemma \ref{lem: C1 lifting} and, since $(M,g)$ is causal geodesically complete we have $x_n \rightarrow \exp_x(w) \in \mathsf O^{+}(x,r)$ so $F(\exp_x(w))= z$ and so the claim follows.
\end{proof}

\begin{remark} \label{rem: countr}
    It is easy to check that the causal geodesic completeness assumption cannot be dropped in Proposition \ref{prop: sumbetry implies submersion}. Consider the restriction on some open diamond of projection map between Minkowski spaces of different dimensions, for instance
    \[
    F: I_{\R^{1,2}}( (1,0,0), (-1,0,0)) \rightarrow I_{\R^{1,1}}( (1,0), (-1,0)) \, .
    \]
    It is clear that $F$ is a smooth and surjective Lorentzian submersion and the source space is globally hyperbolic (and not causal geodesically complete), but both Lemma \ref{lem: C1 lifting} and Proposition \ref{prop: sumbetry implies submersion} fail. \hfill$\blacksquare$
\end{remark}

\subsection{First properties and the Lifting construction}

Now we are able to give the definition of a Lorentzian submetry. As already anticipated in the introduction, this definition does not require any smoothness and so we just give the definition on general causal sets endowed with a time separation function.

\begin{definition} [Lorentzian submetry] \label{def: lorentzian submeetry}
    Let $(X,\ell_X)$ and $(Y, \ell_Y)$ be causal spaces endowed with a time separation function. A surjective function $F: X \rightarrow Y$ is called 
    \begin{itemize}
        \item[a)] future directed Lorentzian submetry if $F( \mathsf{O}_X^+(x, r) ) = \mathsf{O}_Y^+(F(x), r)$ holds for every $x \in X$ and $r \geq 0$ ,
        \item[b)] past directed Lorentzian submetry if $F( \mathsf{O}_X^-(x, r) ) = \mathsf{O}_Y^-(F(x), r)$ holds for every $x \in X$ and $r \geq 0$.
    \end{itemize} 
    Moreover, a map $F: X \rightarrow Y$ is called a Lorentzian submetry provided it is a future directed and a past directed Lorentzian submetry.
\end{definition}

\begin{remark}
    In the next bullet points we collect some basic properties and examples regarding Lorentzian submetries, see Definition \ref{def: lorentzian submeetry}.

\begin{itemize}
    \item[i)] In general, the notions of future directed submetry and past directed submetry are logically independent concepts even for smooth mappings between spacetimes. For instance, consider the set $A:= \{(0,0)\} \cup \{ (1, x) : x \in [-2,-1] \cup [1,2] \}\subset \R^{1,1}$ and pick the projection map $\mathsf{Pr} : I^+(A) \rightarrow (0, \infty)$ onto the first factor. This is a future directed Lorentzian submetry but not a past directed Lorentzian submetry. 
    \item[ii)] (Changing the time orientation) Let $F: M \rightarrow N$ be a map, where $M$ and $N$ are spacetime with time orientation vector fields $X_M$ and $X_N$ respectively. Then $F$ is a future (resp. past) directed Lorentzian submetry if and only if $F$ is a past (resp. future) directed Lorentzian submetry between $M$ and $N$ with orientations induced by $-X_M$ and $-X_N$.
    \item[iii)] It is easy to check that Lorentzian submetries are always $1$-steep and compositions of future (resp. past) Lorentzian submetries are future (resp. past) Lorentzian submetries.
    \item[iv)] Let $(M,g)$ be a strongly causal spacetime and $x \in M$ be a point. Let $C$ be a cone of future directed timelike vectors in $T_xM$ and consider $\mathcal{C}:= C \cap \{v \in T_xM : \|v\|_g=1\}$. Let $0<s<t< \text{InjRad}_x(\mathcal C)$ and define
    \[
    \Omega^{s,t}_x : = \{ \exp_x ( \xi v) : \xi \in (s,t) , v \in \mathcal{C} \} \, \subset M .
    \]
    Then $F(y):= \ell(x,y)$ defined on $F: \Omega^{s,t}_x \rightarrow (s,t)$ is a smooth Lorentzian submersion.
    \item[v)] A classical feature of submetries between metric spaces is that if the source space is a proper metric space then the target space is proper as well. In the setting of Lorentzian submetries, in general if the source space is globally hyperbolic then the target can be even not strongly causal. For instance, consider the quotient map from Minkowski space into the spacelike cylinder $\mathsf Q: \R^{1,1} \rightarrow S^1 \times \R$. This is easily seen to be a Lorentzian submetry since 
    \[
    \mathsf{Q} ( \mathsf O^+_{\R^{1,1}} (x,r))= \mathsf O^+_{S^1 \times \R} (\mathsf Q(x), r) = S^1 \times \R \, .
    \] \hfill$\blacksquare$
\end{itemize} 
\end{remark}

 The following two lemmas discuss the lifting property of timelike geodesics for a given Lorentzian submetry. We start with the proof of the lifting property in the case of uniquely maximizing geodesics and then we prove the Lifting Theorem in the general case.

 The following lemma is inspired by \cite[Lemma 2.1]{BerGui}.

\begin{lemma} \label{lem: lifting submetries} 
    Let $(M,g)$ be a globally hyperbolic spacetime, $(N,h)$ be a strongly causal spacetime and $F:M \rightarrow N$ be a Lorentzian submetry. Let $x \in N$, $y \in F^{-1}(x)$ and $\gamma :[-a,b] \rightarrow N$ be the unique maximizing timelike geodesic between its endpoints such that $\gamma_0=x$. Then, there exists a unique maximizing timelike geodesic $\sigma: [-a,b] \rightarrow M$ such that $\sigma_0= y$, $F \circ \sigma_t = \gamma_t$ for every $t \in [-a,b]$ and $L_g(\sigma) = L_h(\gamma)$.
\end{lemma}

\begin{proof} 
Without loss of generality we assume $\gamma$ to be future-directed and parametrized by $h$-arclength. Suppose for the moment that $a=0$. Notice that $\gamma_b \in \mathsf O^{+}(\gamma_0, b) = \mathsf O^+(F(y), b) = F(\mathsf O^+(y, b))$ and so there exists $z \in \mathsf O^+(y, b)$ such that $F(z) = \gamma_b$. In particular, thanks to the $1$-steepness of $F$ we get that
    \[
    b= \ell_N( \gamma_0, \gamma_b) =\ell_N( F(y), F(z)) \geq \ell_M (y,z) \geq b \, .
    \]
    It follows that there exists a maximizing timelike geodesic $\sigma: [0,b] \rightarrow M$ parametrized by $g$-arclength connecting $y$ to $z$. Observe that, by the $1$-steepness of $F$, we get that $F\circ \sigma$ is a maximizing geodesic connecting $\gamma_0$ to $\gamma_b$, and so uniqueness forces to have $F \circ \sigma_t= \gamma_t$. 
    
    We are left to show that $\sigma$ is the unique maximizing timelike geodesic starting at $y$ such that $F \circ \sigma_t =\gamma_t$ and $L_g (\sigma)= L_h (\gamma)$. Suppose $\sigma^1, \sigma ^2$ are two of such lifted geodesics parametrized by $g$-arclength. If $\varepsilon >0$ sufficiently small we can extend $\gamma$ into a timelike geodesic $\alpha: [-\varepsilon, b] \rightarrow N$ in such a way that $\alpha_t = \gamma_t$ on $t \in [0,b]$ and $\alpha: [-\varepsilon, \varepsilon] \rightarrow N$ is a maximizing timelike geodesic. Then $\alpha_{-\varepsilon} \in \mathsf O^-(\alpha_0, \varepsilon) = F( \mathsf O^-(y, \varepsilon))$, hence we can find $z \in F^{-1}(\alpha_{\varepsilon})$ and $\beta:[-\varepsilon , 0] \rightarrow M$ maximizing timelike geodesic parametrized by $g$-arclength such that $\beta_0= y$ and $\beta_{-\varepsilon}=z$. Then, using the fact that $F$ is $1$-steep and reverse triangle inequality we get
    \[
    2\varepsilon = \ell_N (\alpha_{-\varepsilon}, \alpha_{\varepsilon}) = \ell_N (F( \beta_{\varepsilon}), F(\sigma^1_{\varepsilon})) \geq \ell_M (\beta_{-\varepsilon}, \sigma^1_{\varepsilon}) \geq \ell_M (\beta_{-\varepsilon},y) + \ell_M(y, \sigma^1_{\varepsilon})=2\varepsilon \, .
    \]
    Hence the concatenation of $\beta$ and $\sigma^1$ restricted on $[0, \varepsilon]$ produces a maximizing timelike geodesic in $M$ between $\beta_{-\varepsilon}$ and $\sigma^1_{\varepsilon}$. But, analogous computations lead to the fact that  the concatenation of $\beta$ and $\sigma^2$ restricted on $[0, \varepsilon]$ produces a maximizing timelike geodesic in $M$ between $\beta_{-\varepsilon}$ and $\sigma^2_{\varepsilon}$. Hence the uniqueness follows by the non-branching property of geodesics. 
    
    The case $b=0$ is completely analogous. Now suppose that $a,b \neq 0$. By the previous part of the proof we can find unique liftings of $\gamma :[-a,0] \rightarrow N$ and $\gamma:[0,b] \rightarrow N$ given by $\sigma :[-a, 0] \rightarrow M$ and $\sigma :[0, b] \rightarrow M$. Gluing these two liftings at $t=0$, we get a continuous lifting $\sigma$ for $\gamma$. Moreover, $\sigma$ is a maximizing timelike geodesic since
    \[
    b+a = \ell_N(\gamma_{-a}, \gamma_b) = \ell_N( F(\sigma_{-a}), F(\sigma_{b})) \geq \ell_M (\sigma_{-a}, \sigma_b) \geq \ell_M(\sigma_{-a}, \sigma_0) + \ell_M(\sigma_0, \sigma_b) = b+a \, ,
    \]
    and uniqueness follows from the already established uniqueness in the case $a=0$.  
\end{proof}

The following Theorem is an improvement of the previous Lemma.

\begin{theorem}[Lifting] \label{thm: lifting}
    Let $(M,g)$ be a globally hyperbolic spacetime, $(N,h)$ be a strongly causal spacetime and $F: M \rightarrow N$ be a Lorentzian submetry. Let $\gamma : [0,1] \rightarrow N$ be a timelike geodesic and fix $y \in F^{-1}(\gamma_0)$, then there exists a unique timelike geodesic $\sigma :[0,1] \rightarrow M$ such that $F \circ \sigma_t= \gamma_t$,  $\sigma_0=y$ and $L_g(\sigma) = L_h(\gamma)$. Moreover, if $\gamma$ is maximizing then $\sigma$ is maximizing as well.
\end{theorem}

\begin{proof}
    Consider a partition $0=t_0 < t_1< \dots t_N =1$ of the interval $[0,1]$ in such a way that $\gamma$ is the unique maximizing timelike geodesic when restricted on $[t_{i-1}, t_{i+1}]$ for every $i=1, \dots N-1$. Lemma \ref{lem: lifting submetries} shows that there exists a unique lifted timelike maximizing geodesic $\sigma:[0,t_1] \rightarrow M$ of $\gamma :[0,t_1] \rightarrow N$ such that $\sigma_0=y$. Now, we can iteratively lift the timelike geodesics $\gamma: [t_{i-1}, t_{i+1}] \rightarrow N$ fixing $\sigma_{t_i} \in F^{-1}(\gamma_{t_i})$ for $i=1, \dots N-1$. Notice that the uniqueness part of Lemma \ref{lem: lifting submetries} implies that these lifted geodesics $\sigma: [t_{i-1}, t_{i+1}] \rightarrow M$ glue in a smooth and unique way since are locally maximizing. Hence, this concludes the first part of the proof. Assume that $\gamma$ is maximizing, then for $\varepsilon>0$ sufficiently small $\gamma: [\varepsilon, 1] \rightarrow M$ is uniquely maximizing. Then, choose $N=1$ and $t_1= \varepsilon$ in the previous partition; this yields that $\sigma:[\varepsilon, 1] \rightarrow M$ is maximizing by Lemma \ref{lem: lifting submetries}. The conclusion follows by letting $\varepsilon \rightarrow 0$.
\end{proof}

Under some completeness assumptions, the Lifting Lemma \ref{lem: lifting submetries} implies that submetries are continuous maps between spacetimes and the fibers are closed and acausal.

\begin{proposition} \label{prop: continuity}
    Let $(M,g)$ be a globally hyperbolic spacetime, $(N,h)$ be a strongly causal spacetime and $F:M \rightarrow N$ be a Lorentzian submetry. Then $F$ is continuous and, for every $x \in N$, the fiber $F^{-1}(x)\subset M$ is closed and acausal. 
\end{proposition}

\begin{proof}
    We start by showing the continuity of $F$. Let $y_n \rightarrow y$ in $M$, we claim that $F(y_n) \rightarrow F(y)$ in $N$. Since $N$ is strongly causal, it is sufficient to check that for every geodesically convex normal neighborhood $U:= I(x,z)$ such that $x \ll F(y) \ll z$ with $\diam_h(U) \leq a$ then there exist $\overline{n} \in \N$ such that $x \ll F(y_n) \ll z$ for any $n \geq \overline{n}$. Using Lemma \ref{lem: lifting submetries} we can connect $x$ to $F(y)$ and $F(y)$ to $z$ via unique maximizing timelike geodesics and lift them in order to get that $\tilde{x} \ll y \ll \tilde{z}$ where $F(\tilde{x}) = x$ and $F(\tilde{z})=z$. Since $y_n \rightarrow y$ then there exists $\overline{n}$ such that  $\tilde{x} \ll y_n \ll \tilde{z}$ for every $n \geq \overline{n}$. But now, using the $1-$steepness of $F$ we get $x \ll F(y_n) \ll z$ for any $n \geq \overline{n}$. 
    
    Now we show that the fibers are acausal, the closedness is an easy consequence of the continuity of $F$. Let $y \neq z \in F^{-1}(x)$ and suppose by contradiction that $y\leq z$ in $M$. Let $\overline{x} \in I^+_N(x)$ such that there exists a unique maximizing timelike geodesic $\gamma :[0, \varepsilon ] \rightarrow N$ parametrized by $h$-arclength connecting $x$ to $\overline{x}$. By Lemma \ref{lem: lifting submetries} , there are two liftings $\sigma$ and $\eta$ of $\gamma$ respectively based at $y= \sigma_0$ and $z = \eta_0$. It is clear that $\sigma_1, \eta_1 \in F^{-1}(\overline{x})$ and $y \leq z \ll \eta_1$, hence $\ell_M (y, \eta_1)> \varepsilon$ violating the $1$-steepness of $F$.
\end{proof}

We conclude this section with an application of Lemma \ref{lem: lifting submetries} that will be used in the main theorem of section \ref{sub: regularity}.

\begin{lemma} [Restriction] \label{lem: restriction}
    Let $(M,g)$ be a globally hyperbolic spacetime, $(N,h)$ be a strongly causal spacetime and $F: M \rightarrow N$ be a Lorentzian submetry. Let $U \subset N$ be a causally convex normal neighborhood then $F^{-1}(U) \subset M$ is a causally convex neighborhood and $F: F^{-1}(U) \rightarrow U$ is a Lorentzian submetry.
\end{lemma}

\begin{proof}
    Let $x,y \in F^{-1}(U)$ such that $x\leq y$ in $M$. Let $\sigma:[0,1] \rightarrow M$ be a $g-$causal curve such that $\gamma_0=x$ and $\gamma_1=y$. Since $F$ is $1$-steep, it follows that $\gamma_t:=F\circ \sigma_t$ is a $h$-causal curve in $N$ connecting $F(x), F(y) \in U$. In particular, $\gamma_t \in U$ for every $t \in [0,1]$. This shows that $\sigma_t \in F^{-1}(U)$ for every $t \in [0,1]$. Hence $F^{-1}(U) \subset M$ is causally convex. We are left to show that $F: F^{-1}(U) \rightarrow U$ is a Lorentzian submetry. Given $z \in F^{-1}(U)$ and $r\geq 0$, we have that
    \[
    F\left ( \mathsf O^{+}_{F^{-1}(U)} (z,r) \right ) = F\left ( \mathsf O^{+}_{M} (z,r) \cap F^{-1}(U) \right ) \subseteq \mathsf
    O^+_N(F(z),r) \cap U = \mathsf O^+_U(F(z),r) \, .
    \]
    Vice versa, if $x \in \mathsf O^+(F(z),r) \cap U= F( \mathsf O^+(z,r)) \cap U$ there exists $\overline{x} \in \mathsf O^+(z,r)$ such that $F( \overline{x})=x \in U$, hence 
    $\overline{x}\in \mathsf O^{+}_{F^{-1}(U)} (z,r)$ and so, the conclusion follows.
\end{proof}

\subsection{A regularity theorem} \label{sub: regularity}

The following theorem is the main result of this subsection and will be crucial for the applications. It studies the optimal regularity of Lorentzian submetries between smooth spacetimes. The proof is inspired by \cite[Theorem A]{BerGui}.

\begin{theorem} \label{thm: regularity}
    Let $(M,g)$ be a globally hyperbolic spacetime and $(N,h)$ be strongly causal spacetime. Let $F: M \rightarrow N$ be a Lorentzian submetry. Then $F$ is a $C^{1,1}_{loc}$ Lorentzian submersion.
\end{theorem}

\begin{proof}
     Notice first that $F$ is continuous thanks to Proposition \ref{prop: continuity}, so we need to prove that $F$ is differentiable and its differential is locally Lipschitz with respect to some background Riemannian metrics. We divide the proof in several steps. \\
     \textbf{Step 1} [Dimensional reduction] Let $x \in N$, since $(N,h)$ is strongly causal we can construct a causally convex normal neighborhood $U$ of $F(x)$ and $C^{\infty}$ local coordinates on $U$ given by
    \[
    \Phi : U \ni y \mapsto (\ell_N(x_1,y), \ell_N(x_2,y) , \dots \ell_N(x_n,y)) \in \R^n
    \]
    where $x_1, \dots x_n \in I_N^-(U)$. This construction can be done with a simple application of the Implicit Function Theorem. Moreover, for each $j =1, \dots n$ we can assume that
    \[
    L_j (\cdot ) : = \ell_N(x_j, \cdot) : U \rightarrow (\inf_U L_j, \sup_U L_j) : = (a_j,b_j)
    \]
    is a smooth Lorentzian submetry. In particular, by Lemma \ref{lem: restriction}, we get that $F_j : = F \circ L_j : F^{-1}(U) \rightarrow (a_j, b_j)$ is a Lorentzian submetry with the same regularity of $F$. From now on, we can assume without loss of generality that $N$ is an open interval $(a,b)$ and $M= F^{-1}(a,b)$.
    Suppose that $F(x) = s \in \R$. Using Lemma \ref{lem: lifting submetries} with some $r>0$ sufficiently small, we can construct a lifing $\sigma$ of $[s-r, s+r] \subset N$ such that $\sigma_0=x$ and $F(\sigma_t) = s+t$ for any $t \in [-r,r]$. We define $h_r(x) = \sigma_r$ and $h_{-r}(x) = \sigma_{-r}$, notice that this procedure well defines the maps $h_r$ and $h_{-r}$ on $W:=F^{-1}(a+r, b-r)$. Notice that $h_{-r} \circ h_r= \text{Id}$ holds because of the uniqueness part of Lemma \ref{lem: lifting submetries}. For any $y \in W$, we call 
    \[
    \sigma^r_y :[-r,r] \rightarrow M
    \]
    the unique maximizing timelike geodesic parametrized by $g$-arclength from $h_{-r}(y)$ to $h_r(y)$. \\
    \textbf{Step 2} [The maps $h_r$, $h_{-r}$ are continuous] We show the argument for $h_r$ since the other is analogous. Let $y_n \rightarrow y$ in $W$ and suppose that $F(y_n):= t_n$ and $F(y)=t$, so $t_n \rightarrow t$. Moreover, by construction we have $F(\gamma^n_s) =s+t_n$. Let $\tilde{g}$ be a complete Riemannian metric on $M$ and let $\gamma^n: [0,a_n] \rightarrow M$ be the $\tilde g$-arclength reparameterization of $\sigma_{y_n}^r : [0, r] \rightarrow M$.
    
    By the Limit Curve Theorem \cite[Theorem 2.51]{Min19}, we have that that $\gamma^n$ converges locally uniformly to some continuous causal curve $\gamma$ up to subsequences. 
    Since the fibers of $F$ are acausal by Proposition \ref{prop: continuity}, the curve $\gamma$ necessary leaves the fiber where $y$ belong and so $\gamma _\varepsilon \in F^{-1}(\eta)$ for some $t+r >\eta>t$ and $\varepsilon>0$. For every $n$, denote by $\varepsilon_n \in [0,a_n]$ the unique value such that $F(\gamma^n_{\varepsilon_n})= \eta$. Using a contradiction argument, it is easy to show that $\varepsilon_n \rightarrow \varepsilon$ and so $(\sigma_{y_n}^r)_{\eta-t_n}= \gamma^n _{\varepsilon_n} \rightarrow \gamma_{\varepsilon}$.
    This is sufficient to conclude that $\ell(\gamma_0, \gamma_{\varepsilon})=\eta-t >0$ and that $\gamma$ is timelike. Hence, $\gamma_{\varepsilon}$ is the endpoint of the unique lifting of $[t,\eta]$ based at $y$ and so, by continuity of the exponential map we get also that $h_r(y_n) \rightarrow h_r(y)$. 
    
    In particular, we define the following continuous timelike vector field on $X$ on $W \subset M$ by
    \[
    X_y : = \frac{\log_y( h_r(y))}{r} \, .
    \]
    \textbf{Step 3} [The map $F$ is $C^1$ Lorentzian submersion] To show that $F$ is $C^1$, it is sufficient to show that, for every point $y \in W$, there are two smooth functions $F^+_y, F^-_y$ defined on a neighborhood of $y$ such that:
    \begin{itemize}
        \item[i)] $F^-_y \leq F \leq F^+_y$ holds in a neighborhood of $y$ and equality holds only along $\sigma^r_y$, 
        \item[ii)] $\left (\nabla F^+_y \right )_y = \left (\nabla F^-_y \right )_y = X_y$ .
    \end{itemize}
    Indeed, from ii) we infer that $F^+_y, F^-_y$ share the same partial derivatives at $y$ and from i) we conclude that $F$ is differentiable at $y$ with
    \begin{equation} \label{eq: what is the gradient}
        (\nabla F)_y = X_y \, .
    \end{equation}
    The continuity of $X$ and the arbitrariness of $y \in W$ shows that $F$ is $C^1$ on $W$. Let $y \in W$ and suppose $F(y) = t$, we define
    \begin{equation}
        F^-_y (z) : = F(y) +r - \ell( h_{-r}(y), z) \, , \qquad F^+_y (z) : = F(y) -r + \ell( z, h_{r}(y)) \,.
    \end{equation}
    It is clear from the definition that $F^+_y$ and $F^-_y$ are smooth functions in a neighborhood $V$ of $y$ and that ii) holds. Let $z \in V \subset  I(h_{-r}(y),h_{r}(y))$ and suppose that $\ell(h_{-r}(y),z) : = s \in (0,2r)$. Let $z'$ be the unique point along $\sigma^r_y$ such that $\ell( h_{-r}(y), z) = s$. Then, $F(z') = t-r+s$ by construction and, since $F$ is $1$-steep, we have $F(z) \geq t-r+s$. In particular
    \[
    F(z) \geq F(z') = F(y) +r -s = F(y) +r - \ell( h_{-r}(y), z) = F^-_y(z)  \, ,
    \]
    where equality is attained if and only if $z=z'$. The other inequality and respective rigidity for $F^+_y$ is analogous. Hence, we have proved that $F \in C^1$ on $M$. 
    
    We are left to show that $F$ is a Lorentzian submersion. Fix some $y \in M$ and timelike vector $v \in T_{F(y)}N$ such that $\|v\|_h = 1$. Then, by Lemma \ref{thm: lifting} we can just lift the timelike geodesic $\exp^N_{F(y)}(sv)$ obtaining a uniquely defined timelike geodesic of the form $\exp^M_y(sw)$ on $M$ where $\|w\|_g=1$. The identity in \eqref{eq: what is the gradient}, together with the uniqueness of the lifting show that $\d F_y [w]= v$ and $w \in \ker (\d F_x)^{\perp}$, hence the claim follows by linearity.\\
    \textbf{Step 4} [The map $F$ is locally $C^{1,1}$] Fix a background Riemannian  metric $\tilde{g}$ on $M$. Let $y \in W$ and consider a sufficiently small precompact neighborhood $y \in K \subset W$. By \cite[Theorem 3.6]{McCann}, there exists a constant $C=C(K) \in \R$ such that $F^+_z$ is $C$-semiconvex on $K$ and $F^-_z$ is $C$-semiconcave on $K$ with respect to $\tilde{g}$ for every $z \in K$. In particular, by i) and \cite[Corollary 3.3.8]{cannarsa-sinestrari} we get that $F$ is $C^{1,1}$ when it is restricted to the support of $\sigma^r_z$ in $K$ together with a uniform Lipschitz estimate of the derivative. In particular, the vector field $X$ is locally Lipschitz on $K$ and the conclusion follows by the identity in \eqref{eq: what is the gradient}.
\end{proof}

In the following proposition we study Lorentzian submetries where the target space is the real line $\R$. 

\begin{corollary}\label{cor: unit gradient}
    Let $(M,g)$ be a globally hyperbolic spacetime and $F: M \rightarrow \R$ be a map. If $F$ is a Lorentzian submetry then it is $C^{1,1}_{loc}$ and the gradient $\nabla F$ is complete with unit modulus. Conversely, if $F$ is $C^1$ and $\nabla F$ is a complete vector field such that $\| \nabla F\|_g=1$, then $F$ is a Lorentzian submetry.
\end{corollary} 

\begin{proof}
    Without loss of generality we assume that $\nabla F$ is future directed. Suppose that $F$ is a Lorentzian submetry. Then, following the lines of the proof of Theorem \ref{thm: regularity} we get that $F$ is a $C^{1,1}_{loc}$ Lorentzian submersion and from the identity in \eqref{eq: what is the gradient} we get that $\nabla F$ correspond to the lift of the vector field $\partial_t$ in $\R$. This shows that $\nabla F$ is timelike future directed with unit modulus. Moreover, using Lemma \ref{lem: lifting submetries}, it is easy to check that the gradient flow curves of $F$ coincides with the liftings of $\R$, hence $\nabla F$ is complete. Conversely, if $F$ is a $C^1$ function and $\nabla F$ is complete with unit modulus we claim that
    \[
    F (\mathsf O^+(x,r)) = [F(x)+r, + \infty ) \, , \qquad F (\mathsf O^-(x,r)) = (- \infty , F(x)-r] \, 
    \]
    holds for any choice of $x \in M$ and $r \geq 0$. It is easy to check that $F$ is $1$-steep and so the inclusions $\subseteq$ are trivial. We are left to show the other inclusion. Fix $x \in M$, by Peano's Theorem there exists a solution of 
    \[
    \begin{cases}
        \dot \gamma_t = \nabla F_{\gamma_t} \\
        \gamma_0=x \, .
    \end{cases}
    \]
    Moreover, $\gamma_t$ is well defined for every $t \in \R$ since $\nabla F$ is assumed to be complete and $\gamma$ is a timelike $C^1$ curve satisfying
    \[
    \dfrac{\d}{\d s} \bigg |_{s=t}F(\gamma_s) = \d F_{\gamma_t} [\dot \gamma_t] = \| \nabla F _{\gamma_t}\|^2_g = 1 \, .
    \]
    Hence, using the $1$-steepness of $F$ we get
    \[
    t = F(\gamma_t) - F(x) \geq \ell (x, \gamma_t) \geq \int_0^t \| \dot \gamma_t \|_g  \d t = t \, .
    \]
    Hence $\gamma_t \in \mathsf{O}^+ (x,r)$ for any $t \geq r$ and $\gamma_t \in \mathsf{O}^- (x,r)$ for any $t \leq -r$. Moreover $F(\gamma_t) = F(x)+t$ holds for every $t \in \R$ and so the conclusion follows.
\end{proof}

The following result is a corollary of Theorem \ref{thm: regularity}.

\begin{corollary}
    Let $(M,g)$ be a globally hyperbolic spacetime, $(N,h)$ be a strongly causal spacetime and $F:M \rightarrow N$ be a Lorentzian submetry. Suppose that there exists a $1$-steep smooth map $G: N \rightarrow M$ such that $F \circ G =$ \emph{Id}$_N$. Then $G(N) \subset M$ is an embedded totally geodesic Lorentzian submanifold of $M$. In particular, $(N,h)$ is globally hyperbolic.
\end{corollary}
\begin{proof}
     It is clear that $G$ is a smooth embedding since $F \circ G =$ Id$_N$ and $F$ is a $C^{1,1}_{loc}$ Lorentzian submersion by Theorem \ref{thm: regularity}. Let $\gamma: (a,b) \rightarrow N$ be a timelike geodesic parametrized by $h$-arclength and consider $\sigma:= G\circ \gamma$. It is clear that $\gamma$ is locally maximizing and so, for any $t \in (a,b)$ and $\varepsilon >0$ sufficiently small we have that $\ell_N(\gamma_t, \gamma_{t+ \varepsilon})=\varepsilon$. Morever, since $F,G$ are $1$-steep we have that 
     \[
     \varepsilon= \ell_N(\gamma_t, \gamma_{t+\varepsilon}) \leq \ell_M ( \sigma_t , \sigma_{t+\varepsilon}) \leq \ell_N (F(\sigma_t), F(\sigma_{t+\varepsilon}) ) = \ell_N(\gamma_t, \gamma_{t+\varepsilon})=  \varepsilon \, ,
     \]
     that is to say $\sigma: (a,b) \rightarrow G(N) \subset M$ is a locally maximizing timelike geodesic in $(M,g)$ parametrized by $g$-arclength. This shows that $G(N) \subset M$ is a totally geodesic genuinely Lorentzian smooth submanifold of $M$. Moreover causal diamonds in $G(N)$ are compact since $(M,g)$ is globally hyperbolic, then $(N,h)$ is globally hyperbolic.
\end{proof}

\subsection{Examples}

Given a metric space $(X, \mathsf d_X)$ it is possible to construct a \LpLS by taking a product with $\R$, that is to say $\R \times X$ and
\begin{equation} \label{eq: time sep product}
    \ell_{\R \times X}( (s,x) , (t,y)) = \begin{cases}
        \sqrt{(t-s)^2 - \mathsf d_X (x,y)^2} & t-s \geq \mathsf d_X (x,y) , \\
        0 & \text{otherwise}
    \end{cases}
\end{equation}
Given some $(t,x) \in \R \times X$ and $r \geq 0$, an easy algebraic manipulation yields 
\begin{equation}
    \mathsf O^+_{\R \times X} ( (s,x),r ) : = \bigcup_{t \geq s+r} \{t \} \times \mathsf B_{X}\left (x, \sqrt{(t-s)^2-r^2} \right ) \, ,
\end{equation}
where $\mathsf B_{X}(x,r)$ denotes the metric ball centered at $x \in X$ with radius $r$. Analogous formulas hold for $\mathsf O^-$ as well. In particular, we get the following proposition.

\begin{proposition} \label{prop: easy}
    Let $(X, \mathsf d_X)$ and $(Y, \mathsf d_Y)$ be metric spaces and $f: X \rightarrow Y$ be map. Consider the product \LpLSs given by $\R \times X$ and $\R \times Y$
    endowed with the respective product time separation functions \eqref{eq: time sep product}. Then the function 
    \[
    F: \R \times X \rightarrow \R \times Y \, \qquad \, F(t,x)= (t, f(x))
    \]
    is a Lorenztian submetry if and only if $f$ is a submetry.
\end{proposition}

It might be useful to recall that if $(M,g)$ is a Riemannian manifold then the time separation function defined in \eqref{eq: time sep product} coincides with the one of the product spacetime $\R \times M$ with the metric $\d t^2-g$. Proposition \ref{prop: easy} permits to construct a first example of a Lorentzian sumbetry between smooth globally hyperbolic and geodesically complete spacetimes that is locally $C^{1,1}$ but not $C^2$. 

The following example shows that the regularity obtained in Theorem \ref{thm: regularity} is the best one can get without further assumptions.

\begin{example}
    In \cite[Section 4]{BerGui}, it is explicitly constructed a submetry $f: D^2 \rightarrow \R$ that is locally $C^{1,1}$ but not $C^2$. This is done by choosing an equidistant foliation whose level sets are not $C^2$ of the Poincaré disk $(D^2, g)$. Now consider $F: \R \times D^2 \rightarrow \R^{1,1}$
    given by $F(t,x) = (t, f(x))$ and endow $\R \times D^2$ with the product metric $dt^2 - g$ obtaining a globally hyperbolic and geodesically complete spacetime. It is clear that and $F$ is locally $C^{1,1}$ but not $C^2$ and, thanks to Proposition \ref{prop: easy} we have that $F$ is a Lorentzian submetry.
\end{example}

In the following example we construct explicitly another submetry that is not $C^2$ in the space of constant negative Timelike sectional curvature, namely De Sitter spacetime.

\begin{example} \label{example desitter}
    We consider the local coordinates $t,x$ describing the $2$-dimensional De Sitter spacetime by
\begin{equation}
    \mathsf{dS}^2 : = \left ( - \frac{\pi}{2} , \frac{\pi}{2} \right ) \times S^1 \, , \qquad g_{\mathsf{dS}^2}= \frac{1}{\cos(t)^2} \left ( \d t^2 - \d x^2 \right ) \, ,
\end{equation}
where $S^1$ is the standard circle parametrized by $x \in [-\pi , \pi)$.
It is easy to check that
\[
\gamma : (- \infty , \infty ) \rightarrow M \, , \qquad  \quad \gamma_s : = \left ( 2 \arctan \left ( e^s \right ) - \frac{\pi}{2} , 0 \right )
\]
is a timelike line. The time separation function on $\mathsf{dS}^2$ can be computed explicitly, see \cite[Lemma 2.8]{BeranSaman} and an elementary computation shows that the Busemann functions associated to this line $b^-_{\gamma} : I^{-}(\gamma) \rightarrow \R$ and $b^+_{\gamma} : I^{+}(\gamma) \rightarrow \R$ are smooth on their domain and given by
\[
b^-_{\gamma}(t,x) : = \lim_{s \rightarrow \infty} s-\ell_{\mathsf{dS}^2} ((t,x), \gamma_s) = -\log \left ( \frac{\cos(x) - \sin (t)}{\cos(t)} \right ) \, ,
\]
\[
b^+_{\gamma}(t,x) : = \lim_{s \rightarrow -\infty} s-\ell_{\mathsf{dS}^2} ((t,x), \gamma_s) = -\log \left ( \frac{\cos(x) - \sin (-t)}{\cos(-t)} \right ) \, ,
\]
where
\[
I^{-}(\gamma) = \left \{ (t,x) \in \mathsf{dS}^2 : t < \frac{\pi}{2} - |x| \right \} \, , \qquad I^{+}(\gamma) = \left \{ (t,x) \in \mathsf{dS}^2 : t > -\frac{\pi}{2} + |x| \right \} \, .
\]
The Lorentzian gradient $\nabla b^-_{\gamma}$ is a complete future-directed timelike geodesic vector field whose integral lines are maximizing geodesics connecting an ideal point in $\{t = - \pi /2\}$ to $(\pi/2,0)$. Moreover, a direct computation shows that $\| \nabla b^-_{\gamma} \|_g=1$ holds in $I^-(\gamma)$ and so, by Corollary \ref{cor: unit gradient} we get that $b^-_{\gamma} : I^{-}(\gamma) \rightarrow \R$ is a smooth Lorentzian submetry, see the following picture for the foliation of $I^{-}(\gamma)$ obtained by the level sets of $b^-_{\gamma}$.

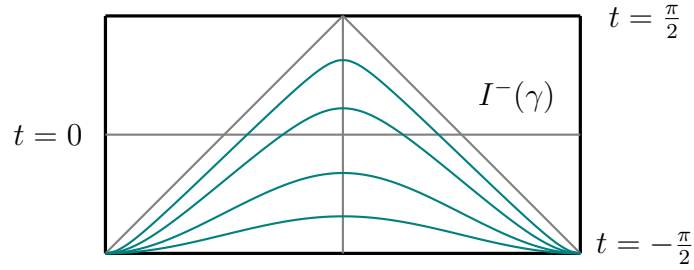
\begin{center}
\begin{tikzpicture}

\draw (4,1.5) node {$t= \frac{\pi}{2}$};
\draw (-3.9,0) node {$t= 0$};
\draw (4,-1.5) node {$t= -\frac{\pi}{2}$};
\draw (2.3,0.5) node {$I^{-}(\gamma)$};
\draw[very thick] (-3.14,-1.57) -- (-3.14,1.57);
\draw[very thick] (3.14,-1.57) -- (3.14,1.57);
\draw[very thick] (-3.14,-1.57) -- (3.14,-1.57) ;
\draw[very thick] (-3.14,1.57) -- (3.14,1.57) ;
\draw[thick, gray] (0,1.57) -- (-3.14,-1.57) ;
\draw[thick, gray] (-3.14,0) -- (3.14,0) ;
\draw[thick, gray] (0,1.57) -- (0,-1.57) ;
\draw[thick, gray] (0,1.57) -- (3.14,-1.57) ;

% Plot function
\draw[thick,teal] 
    plot[domain=-3.14:3.14,
        samples = 50,
        smooth]({\x},  {(pi/180) * asin(
            cos(\x r)/sqrt(1 + 0.3^2)
        )
        -
        (pi/180) * atan(0.3)
        });

\draw[thick,teal] 
    plot[domain=-3.14:3.14,
        samples = 50,
        smooth]({\x},  {(pi/180) * asin(
            cos(\x r)/sqrt(1 + 0.7^2)
        )
        -
        (pi/180) * atan(0.7)
        });

\draw[thick,teal] 
    plot[domain=-3.14:3.14,
        samples = 50,
        smooth]({\x},  {(pi/180) * asin(
            cos(\x r)/sqrt(1 + 1.7^2)
        )
        -
        (pi/180) * atan(1.7)
        });

\draw[thick,teal] 
    plot[domain=-3.14:3.14,
        samples = 50,
        smooth]({\x},  {(pi/180) * asin(
            cos(\x r)/sqrt(1 + 4^2)
        )
        -
        (pi/180) * atan(4)
        });
 
\end{tikzpicture}
\end{center}
Analogous properties hold for $b^+_{\gamma}$. Now we combine the two Busemann functions in order to get a non-smooth Lorentzian submetry. Define the map $F: \Omega \rightarrow \R$ by
\[
F(t,x) : = \begin{cases}
b^-_{\gamma}(t,x) & x \leq 0 \\
    -b^+_{\gamma} (t,x) & x \geq 0 \, .
\end{cases} \qquad \Omega : = \left \{ (t,x) \in \mathsf{dS}^2 : x- \frac{\pi}{2} < t < x + \frac{\pi}{2} \right \} \, .
\]
See the following picture for the foliation of $\Omega$ by the level sets of $F$.
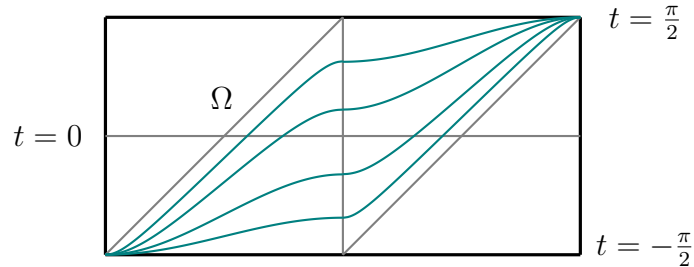
\begin{center}
\begin{tikzpicture}

\draw (4,1.5) node {$t= \frac{\pi}{2}$};
\draw (-3.9,0) node {$t= 0$};
\draw (4,-1.5) node {$t= -\frac{\pi}{2}$};
\draw (-1.6,0.5) node {$\Omega$};
\draw[very thick] (-3.14,-1.57) -- (-3.14,1.57);
\draw[very thick] (3.14,-1.57) -- (3.14,1.57);
\draw[very thick] (-3.14,-1.57) -- (3.14,-1.57) ;
\draw[very thick] (-3.14,1.57) -- (3.14,1.57) ;
\draw[thick, gray] (0,1.57) -- (-3.14,-1.57) ;
\draw[thick, gray] (0,-1.57) -- (3.14,1.57) ;
\draw[thick, gray] (-3.14,0) -- (3.14,0) ;
\draw[thick, gray] (0,1.57) -- (0,-1.57) ;

% Plot function
\draw[thick,teal] 
    plot[domain=-3.14:0,
        samples = 50,
        smooth]({\x},  {(pi/180) * asin(
            cos(\x r)/sqrt(1 + 0.3^2)
        )
        -
        (pi/180) * atan(0.3)
        });

\draw[thick,teal] 
    plot[domain=-3.14:0,
        samples = 50,
        smooth]({\x},  {(pi/180) * asin(
            cos(\x r)/sqrt(1 + 0.7^2)
        )
        -
        (pi/180) * atan(0.7)
        });

\draw[thick,teal] 
    plot[domain=-3.14:0,
        samples = 50,
        smooth]({\x},  {(pi/180) * asin(
            cos(\x r)/sqrt(1 + 1.7^2)
        )
        -
        (pi/180) * atan(1.7)
        });

\draw[thick,teal] 
    plot[domain=-3.14:0,
        samples = 50,
        smooth]({\x},  {(pi/180) * asin(
            cos(\x r)/sqrt(1 + 4^2)
        )
        -
        (pi/180) * atan(4)
        });

\draw[thick,teal] 
    plot[domain=0:3.14,
        samples = 50,
        smooth]({\x},  {-(pi/180) * asin(
            cos(\x r)/sqrt(1 + 3.3^2)
        )
        +
        (pi/180) * atan(3.3)
        });

\draw[thick,teal] 
    plot[domain=0:3.14,
        samples = 50,
        smooth]({\x},  {-(pi/180) * asin(
            cos(\x r)/sqrt(1 + 1.43^2)
        )
        +
        (pi/180) * atan(1.43)
        });

\draw[thick,teal] 
    plot[domain=0:3.14,
        samples = 50,
        smooth]({\x},  {-(pi/180) * asin(
            cos(\x r)/sqrt(1 + 0.59^2)
        )
        +
        (pi/180) * atan(0.59)
        });

\draw[thick,teal] 
    plot[domain=0:3.14,
        samples = 50,
        smooth]({\x},  {-(pi/180) * asin(
            cos(\x r)/sqrt(1 + 0.25^2)
        )
        +
        (pi/180) * atan(0.25)
        });
 
\end{tikzpicture}
\end{center}
It is easy to check that $F$ is not $C^2$ because of the jump discontinuity of the second derivative along $\gamma$. Moreover, the timelike vector field $\nabla F$ is complete with unit modulus so, by Corollary \ref{cor: unit gradient}, we have that $F$ is a Lorentzian submetry.
\end{example}

\section{Applications}

\subsection{Timelike curvature bounds}

Let $(M,g)$ be a spacetime and consider a timelike $2$-plane $\Pi$ in $T_xM$, that is to say a $2$-dimensional subspace such that $g_x |_{\Pi^2}$ has signature $(+,-)$. Suppose that $\Pi = \text{span}(v,w)$ where $v,w \in T_xM$ are timelike vectors, then the sectional curvature of $\Pi$ is defined by
\begin{equation} \label{eq: sectional curvature}
\Sec _x(\Pi) : = \frac{g( R(v,w) \, w, v)}{g(v,w)^2 -\|v\|_g^2 \|w\|_g^2} \, .
\end{equation}
It is easy to check that the definition is well posed and independent of the choice of the generators $v,w$ of $\Pi$. Given some real number $K$, we say that $(M,g)$ has Timelike sectional curvature bounded below (resp. above) by $K$, or $\text{TSec} \geq K$ (resp. $\text{TSec} \leq K$) in short, provided $\Sec _x(\Pi) \geq K$ (resp. $\Sec _x(\Pi) \leq K$) for every timelike $2$-plane $\Pi$. 

In the proof of next results we will make use of the Lorentzian Hinge Comparison, see \cite[Theorem 3.2]{BKOR}. We will need also the strict Hinge Comparison, see \cite[Proposition 4.2 \& Corollary 4.3]{Ptolemy}.

\begin{proposition} \label{prop: sectional descends}
    Let $(M,g)$, $(N,g)$ be globally hyperbolic spacetimes, $F: M \rightarrow N$ be a Lorentzian submetry and $K \in \R$. If $\emph{TSec}_M \geq K$ then $\emph{TSec}_N \geq K$. Moreover, if $\emph{TSec}_M \geq K$ and $\emph{TSec}_N = K$, then $\emph{TSec}_M = K$.
\end{proposition}

\begin{proof}
    Suppose by contradiction that there exists a point $x_0 \in N$ and a timelike $2$-plane $\Pi \subset T_{x_0}N$ such that $\text{Sec}_{x_0}^N ( \Pi ) < K$. Fix two timelike generators of $\Pi$, say $v_0, w_0 \in T_{x_0}N$ and up to rescale these two vectors to get a sufficiently small Lorentzian norm, we can assume that 
    \[
    y_0: =\exp^N_{x_0}( v_0) \ll z_0: =\exp^N_{x_0}( w_0)
    \]
    and both $y_0,z_0$ are contained in a (strict) hinge $K$-comparison neighborhood of the point $x_0$ in $N$. The maximizing timelike geodesics $\gamma^1_t : = \exp^N_{x_0}(t v_0)$ and $\gamma_t^2 : = \exp^N_{x_0}(t w_0)$ form an hinge based at the point $x_0$. Consider a comparison hinge in the model Lorentzian space $\mathbb L^2(K)$ of constant sectional curvature $K$ corresponding to the two timelike geodesics $\tilde \gamma^1, \tilde \gamma^2$. Then the strict Hinge Comparison Theorem tells that
    \begin{equation} \label{eq: hinge 1}
        \ell_N (y_0,z_0) < \ell_{\mathbb L^2(K)} ( \tilde{\gamma}^1_1, \tilde{\gamma}^2_1) \, .
    \end{equation}
    Now, consider a point $x \in F^{-1}(x_0)$ and let $\sigma^1, \sigma^2$ be the unique lifting of $\gamma^1$, $\gamma^2$ respectively based at $x$ in the sense of Lemma \ref{lem: lifting submetries}. It follows from the construction that $\sigma^1, \sigma^2$ form an hinge in $M$ at $x$. Let $y: = \sigma^1_1$, $z: = \sigma^2_1$ and
    \[
    v : = \log_{x} (\sigma^1_1) \, , \qquad w : = \log_{x} (\sigma^2_1)
    \]
    As a consequence of Theorem \ref{thm: regularity}, we get that $F$ is a locally $C^{1,1}$ Lorentzian submersion, $v,w \in \ker( \d F_x)^{\perp}$, as well as $\d F_x [v] = v_0$, $\d F_x [w] = w_0$ and so $g(v,w) = h ( v_0, w_0)$. Hence the angle between $\sigma^1$ and $\sigma^2$ is the same as the angle between $\gamma^1$ and $\gamma^2$, and so the hinge in the model space $\mathbb L^2(K)$ is a comparison hinge also for the one just constructed in $M$. Since $\text{TSec}_M \geq K$ we have that
    \begin{equation} \label{eq: hinge 2}
        \ell_M( y,z) \geq \ell_{\mathbb L^2(K)} ( \tilde{\gamma}^1_1, \tilde{\gamma}^2_1) \, .
    \end{equation}
    But the combination of \eqref{eq: hinge 1} and \eqref{eq: hinge 2} gives is a contradiction with the $1$-steepness of $F$ since $F(y)= y_0$ and $F(z)= z_0$. Suppose now that $\text{TSec}_N = K$. Following the previous computations we get that
    \[
    \ell_{\mathbb L^2(K)} ( \tilde{\gamma}^1_1, \tilde{\gamma}^2_1) = \ell_N (F(x), F(y) )\geq \ell_M( y,z) \geq \ell_{\mathbb L^2(K)} ( \tilde{\gamma}^1_1, \tilde{\gamma}^2_1) \, ,
    \]
    and this shows that $\text{TSec}_M = K$ by \cite[Theorem 3.1]{BKOR}, concluding the final claim.
\end{proof}

Now we consider applications in the case of Ricci curvature bounds. For completeness, we recall the statement of the Lorentzian splitting Theorem, see \cite{Splitting} and \cite{ellipticsplitting}.

\begin{theorem} [Lorentzian Splitting Theorem] \label{thm: splitting}
    Let $(M,g)$ be a globally hyperbolic spacetime satisfying \eqref{SEC}. Suppose that $M$ contains a timelike line, then there exists a Riemannian hypersurface $X \subset M$ such that $(X, -g |_{X})$ is a complete Riemannian manifold with non-negative Ricci curvature and there exists a $C^{\infty}$ Lorentzian isometry map 
    \[
    \Phi : (M,g) \rightarrow ( \R \times X, \d t^2 + g|_{X} ) \, .
    \]
\end{theorem}

Now we provide some applications of Lorentzian submetries between spacetimes satisfying the Strong Energy Condition \eqref{SEC}. We start with the case when the target space is $\R$.

\begin{proposition} \label{prop: appl Ricci 1}
    Let $(M,g)$ be a globally hyperbolic spacetime satisfying \eqref{SEC}. Let $F: M \rightarrow \R$ be a Lorentzian submetry. Then there exists a Riemannian hypersurface $X$ such that $M$ is isometric to $\R \times X$. In particular, $F$ is smooth and coincides with the projection to the first factor in $F: M \simeq \R \times X \rightarrow \R$.
\end{proposition}

\begin{proof}
    Using Corollary \ref{cor: unit gradient} we get that $F$ is $C^{1,1}_{loc}$ and $\nabla F$ is a complete vector field with unit modulus. Moreover, each gradient flow curve of $\nabla F$ coincides with the lifting of $\R$ in the sense of Lemma \ref{lem: lifting submetries} and so, these curves must be timelike lines. Since $\Ric \geq 0$ along causal vectors and a timelike line exists, the Lorentzian splitting Theorem ensures that $M$ splits isometrically as $\R \times X$ where $X \subset M$ is a totally geodesic Riemannian hypersurface in $M$. Fix now $(0, \xi) \in \R \times X \simeq M$, then, by uniqueness, the integral curve of $\nabla F$ starting at $(0, \xi)$ needs to be $\gamma_t = (t,\xi)$ hence $F(t,\xi) =t$ and so the conclusion follows.
\end{proof}

The following theorem is an improvement of the previous proposition.

\begin{theorem}
    Let $(M,g)$ and $(N,h)$ be globally hyperbolic spacetimes satisfying \eqref{SEC} and $F: M \rightarrow N$ be a Lorentzian submetry. Suppose that $N$ contains a timelike line, then there exists an integer $k \in \mathbb N$ and Riemannian submanifolds $X \subset M$, $Y \subset N$ such that $Y$ do not contain any line, $M$ is isometric to $\R^{1,k} \times X$, $N$ is isometric to $\R^{1,k} \times Y$ and $F$ can be identified with 
    \[
    F= (\emph{Id}, f) : \R^{1,k} \times X \rightarrow \R^{1,k} \times Y \, ,
    \]
    and $f : X \rightarrow Y$ is a submetry and a locally $C^{1,1}$ Riemannian submersion between complete Riemannian manifolds.
\end{theorem}

\begin{proof}
    Since $N$ contains a timelike line and satisfies \eqref{SEC}, then, by the Lorentzian Splitting Theorem \ref{thm: splitting} there exists a Riemannian hypersurface $N_1 \subset N$ such that $N \simeq \R \times N_1$. Let $\mathsf{Pr}_1 : N \simeq \R \times N_1 \rightarrow \R$ be the projection onto the first factor. Then $F \circ \mathsf{Pr}_1: M \rightarrow \R$ is a Lorentzian submetry and so, by Proposition \ref{prop: appl Ricci 1} we have the isometric splitting $M \simeq \R \times M_1$ where $\mathsf{Pr}_1 \circ F( t, x) = t$ and $M_1$ is a Riemannian manifold with non-negative Ricci curvature. Since $M,N$ are globally hyperbolic, then $M_1,N_1$ are proper metric spaces and so are also complete Riemannian manifolds by Hopf-Rinow Theorem. Now we have that
    \[
    F : \R \times M_1 \rightarrow \R \rightarrow N_1 \, \qquad F(t, x ) = (t, F_1(t, x)) \, ,
    \]
    where $F_1 : = F_1 (t,x)$ correspond to the composition of $F$ with the projection onto the second component in the target $\R \times N$. We claim that $F_1$ do not depend on $t$ and it is a submetry between complete Riemannian manifolds. Notice that $s \mapsto (s, x)$ is a timelike line that arises as the unique lifting based at $(0,x) \in M$ in the sense of Lemma \ref{lem: lifting submetries} of a timelike line $s \mapsto (s, y)$ in $N$ with respect to $F$ for some $y\in N_1$ such that $F_1(t,x)= y$. Then $F_1(s,x)=y$ by construction for any $s \in \R$. Hence $F_1 : M_1 \rightarrow N_1$ is a map between complete Riemannian manifolds and, as a consequence of Proposition \ref{prop: easy} we get that $F_1$ is a submetry. If $N_1$ do not contain any line, the conclusion follows and $k=0$. If $N_1$ contains a line, the Riemannian Splitting Theorem of \cite{RiemSplitting} coupled withe previous arguments yields the conclusion after a finite number of iterations.
\end{proof}

\subsection{Equidistant acausal foliations}

In this section we study foliations of globally hyperbolic spacetimes. We show that Lorentzian submetries are intimately related to the notion of equidistant acausal foliations, see Definition \ref{def: foliation}. We start with the definition of equidistant acausal foliations in the setting causal spaces endowed with a time separation function.

\begin{definition}[Foliations] \label{def: foliation}
    Let $(X,\ell)$ be a casual space endowed with a time separation function. A partition $\mathcal{F}$ of $X$ into acausal subsets is called an acausal foliation. Furthermore, we say that the acausal foliation $\mathcal{F}$ is equidistant provided for every $A,B \in \mathcal{F}$ and every $z \in A$ and $w \in B$ we have that
    \begin{equation} \label{eq: equidistant}
        \ell(A,B) =\ell(z,B)= \ell(A,w) \, .
    \end{equation}
\end{definition}

We remark that the suprema in the defining time separations in the identity \eqref{eq: equidistant} are in general not attained. Let $X$ be a causal space endowed with a time separation and $\mathcal{F}$ be an equidistant acausal foliation. Then we say $x \sim y$ in $X$ if and only if $x,y$ belong to the same element in the partition given by $\mathcal{F}$. It is easy to check that $\sim$ is an equivalence relation, and so the time separation function $\ell_X$ descends to the quotient 
\begin{equation} \label{eq: quotient}
    X^{\ast}: =X / \sim
\end{equation}
and it is well defined due to the requirement in \eqref{eq: equidistant}. More precisely, let $A,B \in X^{\ast}$, we say that $A \leq^{\ast} B$ provided there exists some $x\leq y$ in $X$ such that $[x]=A$ and $[y]=B$. Analogously, we say that $A \ll^{\ast} B$ provided there exists some $x\ll y$ in $X$ such that $[x]=A$ and $[y]=B$ Moreover, we put $\ell^{\ast}(A,B) = \ell(A,B)$ on $X^{\ast}$. 

\begin{lemma} \label{lem: quotient}
    $X^{\ast}$ is a causal space and $\ell^{\ast}$ is a time separation function. Moreover if $\leq$ is a partial order in $X$, then $\leq^{\ast}$ is a partial order on $X^{\ast}$. 
\end{lemma}

\begin{proof}
    The unique non-trivial thing to check is that $\leq^{\ast}$ is a partial order on $X^{\ast}$. Let $A,B \in X^{\ast}$ and suppose that $A \leq^{\ast} B$ and $B \leq^{\ast} A$. Choose some representatives $x \leq y$ such that $[x]=A$, $[y]=B$. Moreover, $A \leq ^{\ast} B$ together with \eqref{eq: equidistant} implies that $\ell(A, y) \geq 0$ that is to say that there exists $x_1 \in X$ such that $[x_1] = A$ and $x_1 \leq y$. Hence, $x \leq y \leq x_1$ but $A$ is acausal and so $x=x_1$. Moreover, $\leq$ is a partial order on $X$ and so $x \leq y \leq x$ implies that $x=y$ and so $A=B$.
\end{proof}

The following result shows that equidistant acausal foliation are in one-to-one correspondence with Lorentzian submetries up to $\ell$-preserving maps. This is a generalization in the Lorentzian setting of the result \cite[Lemma 8.4]{QuotientMondino}.

\begin{lemma} \label{lem: foliation}
    Let $(M,g)$ be a globally hyperbolic spacetime, $(X, \ell_X)$ a causal set with a time separation and $F:M \rightarrow X$ be a Lorentzian submetry. Then, the collection $\{ F^{-1}(x) : x \in X \}$ is an equidistant acausal foliation. On the other side, if $\mathcal{F}$ is an equidistant acausal foliation of $M$, then the natural quotient map $\mathsf Q : M \rightarrow M^{\ast}$ is a Lorenztian submetry, where $M^{\ast}$ has been defined in \eqref{eq: quotient}. Moreover, this is a one-to-one correspondence up to $\ell$-preserving maps.
\end{lemma}

\begin{proof}
Suppose that $F:M \rightarrow X$ is a Lorentzian submetry. Then we claim that $\{F^{-1}(x) : x \in X\}$ is an equidistant acausal foliation. Let $x \in X$, suppose by contradiction that $z\leq z'$ in $F^{-1}(x) \subset M$ with $z\neq z'$. Pick any $w \in I_M^+(z)$ and, since $F$ is $1$-steep we get $y: = F(w) \in I^+_X(x)$, i.e. $\ell_X(x,y)=r$ with some $r>0$. Hence
\[
y \in F(\mathsf O^+_X(z,r)) \qquad \text{and} \qquad  y \in F(\mathsf O^+_X(z',r)) \, , 
\]
and so, there exist some $v,v' \in F^{-1}(y)$ such that $\ell_M (z,v), \ell_M (z',v') \geq r$. The $1$-steepness of $F$ actually implies that $\ell_M (z,v)= \ell_M (z',v') = r$. But now $z\leq z' \ll v'$ and, since $M$ is globally hyperbolic, we have that $\ell_M (z, v') >r$ and this is a contradiction with the $1$-steepness of $F$.\\
Now we show that the fibers are equidistant in the sense of \eqref{eq: equidistant}. Let $x,y \in X$ and pick some $v \in F^{-1}(x), w \in F^{-1}(y)$, if $x \not \leq y$ then $\ell_M( F^{-1}(x), F^{-1}(y))= \ell_M(v,F^{-1}(y))= \ell_M (F^{-1}(x), w)= - \infty$ because of the $1$-steepness of $F$. Suppose $x \leq y$, suppose by contradiction that $a: =\ell_M(F^{-1}(x), F^{-1}(y)) > \ell_M (v, F^{-1}(y)): = b \geq 0$. Notice that the other inequality was trivially verified by definition. Then, there exists $\tilde{v}, \tilde{w}$ such that $\tilde{v} \in F^{-1}(x)$, $\tilde{w} \in F^{-1}(y)$ and $\ell_X(x,y) \geq \ell_M (\tilde{v}, \tilde{w}) = a-\varepsilon >b$, for any  $\varepsilon>0$. Hence
\[
y \in \mathsf O_N^+(x, a-\varepsilon) = \mathsf O_N^+(F(v), a-\varepsilon) = F(\mathsf O_M^+(v, a-\varepsilon)) \, ,
\]
and so by surjectivity of $F$ we get that actually $\ell_M (v, F^{-1}(y)) \geq a-\varepsilon$. By arbitrariness of $\varepsilon>0$ we get that $\ell_M (F^{-1}(x), F^{-1}(y))= \ell_M(v, F^{-1}(y))$, the case of $w$ is completely analogous.
Let $\mathcal{F}$ be an equidistant acausal foliation of $M$. Lemma \ref{def: foliation} actually provides that $(M^{\ast}, \ell)$ is a well defined causal space with a time separation. It is easy to check that the quotient map is a Lorenztian submetry and the map $\iota_F = \mathsf Q \circ F^{-1} : X \rightarrow M^{\ast}$ is an $\ell$-preserving map. We are left to show the one-to-one correspondence up to $\ell$-preserving maps. Suppose there are two submetries $F: M \rightarrow X$ and $G:M \rightarrow Y$ that induces the same equidistant acausal foliation of $M$, then $\iota_G^{-1} \circ \iota_F : X \rightarrow Y$ provides an $\ell$-preserving map between $X$ and $Y$. 
\end{proof}

\subsection{Isometric group actions}

In this subsection we apply the theory of Lorentzian submetries and equidistant acausal foliations in the case of isometric group actions in analogy to the Riemannian case studied in \cite{QuotientMondino} and \cite[Section 5.5]{LottVillani}. Let $(M,g)$ be a spacetime and $\mathsf G$ be a compact Lie group acting smoothly on $M$ by $g$-isometries. We denote by $\tau_{\delta}: M \rightarrow M$ the $g$-isometry induced by $\delta \in \mathsf G$. Let $x \in M$, we define its $\mathsf G$-orbit by
\[
[x]: = \{\tau_{\delta}(x) : \delta \in \mathsf G\} \, .
\]
Suppose that the orbits are acausal, then it is easy to check that the collection of orbits is an equidistant acausal foliation in the sense of Definition \ref{def: foliation}. Let $M / \mathsf G$ be the quotient space and denote by $\mathsf{Q}: M \rightarrow M/ \mathsf G$ the quotient map. Accordingly to Lemma \ref{lem: quotient} the quotient space is a causal space endowed with a time separation function $\ell_{M/\mathsf G}$. The following lemma shows that actually the quotient space is a Lorentzian pre-lenght space.

\begin{lemma}
    Let $(M,g)$ be a globally hyperbolic spacetime and $\mathsf G$ be a compact Lie group acting by $g$-isometries on $M$ with acausal orbits. Then the family of orbits of the action is an equidistant acausal foliation of $M$. Moreover, $(M/\mathsf G, \ell_{M/ \mathsf G})$ is a well defined Lorentzian pre-length space, $\leq_{M/ \mathsf G}$ is a partial order and the quotient map $\mathsf Q$ is a Lorentzian submetry. 
\end{lemma}

\begin{proof}
    By \cite[Theorem 0.2]{Marja}, $M$ admits a complete $\mathsf G$-invariant Riemannian metric and it induces a $\mathsf G$-invariant proper distance $\mathsf d$ on $M$. Hence, the quotient metric space $(M/\mathsf G, \mathsf d_{M/\mathsf G})$ is a proper metric space and $\mathsf Q$ is a submetry between proper metric spaces. The partition $\mathcal{F}$ given by the orbits is an equidistant acausal foliation consisting in compact subsets of $M$ and so the time separation function $\ell_{M/\mathsf G}$ is well defined and continuous. The remaining properties actually follows by Lemma \ref{lem: quotient}.
\end{proof}

Now we want to show that if $(M,g)$ satisfies a lower bound on the Ricci curvature on causal vectors then also the quotient space $M/\mathsf G$ satisfies a lower bound of the Ricci curvature in a synthetic sense. In order to do that we have to study the Lorentz-Wasserstein space of probability measures. 

From now on, we endow the spacetime $(M,g)$ with its volume measure $\Vol_g$ and the quotient $M/ \mathsf G$ with the push-forward measure $\mathsf Q_{\#} \Vol_g$. Moreover, absolute continuity is understood with respect to these reference measures.

\begin{definition}
    Let $(M,g)$ be a spacetime and $\mathsf G$ be a compact Lie group acting smoothly on $M$ by $g$-isometries. We say that a probability measure $\mu \in \Prob(M)$ is $\mathsf G$-invariant provided $(\tau_{\delta}) _{\#} \mu = \mu$ for every $\delta \in \mathsf G$ and we denote by $\Prob^{\mathsf G}(M)$ such a set of $\mathsf G$-invariant probability measures.
\end{definition}

Let $\mathfrak{h}$ be the Haar probability measure of $\mathsf G$. Fix a point $x \in M$ and consider the smooth map $\mathsf G \ni \delta \mapsto \tau^x(\delta) : =\tau_{\delta}(x) \in M$. We define 
\[
\mathfrak{h}_{[x]}: = (\tau^x)_{\#} \mathfrak{h} \in \Prob^{\mathsf G}(M) \, .
\]
Each measure $\mathfrak{h}_{[x]}$ is concentrated on the compact set $[x] \subset M$ and if $[x]=[y]$ then $\mathfrak{h}_{[x]}= \mathfrak{h}_{[y]}$ i.e. the collection does not depend on the representative in $M/ \mathsf G$. In analogy to \cite{QuotientMondino}, we define $\Lambda: \Prob(M/\mathsf G) \rightarrow \Prob^{\mathsf G}(M)$ in the following way: let $\nu \in \Prob (M/\mathsf G)$ and define
\begin{equation}
    \Lambda (\nu) : = \int \mathfrak{h}_{[x]} \, \nu(\d[x]) \in \Prob^{\mathsf G} (M)
\end{equation}
The following theorem is a natural generalization of \cite[Theorem 3.2]{QuotientMondino}.
\begin{theorem} \label{thm: analogo mondino}
    Let $q \in (0,1)$, let $(M,g)$ be a globally hyperbolic spacetime and $\mathsf G$ be a compact Lie group acting by $g$-isometries on $M$ with acausal orbits. The map $\Lambda: \Prob(M/ \mathsf G) \rightarrow \Prob^{\mathsf G}(M)$ is well defined and, for any $\nu \in \Prob(M/\mathsf G)$, then $\Lambda(\nu)$ is the unique $\mathsf G$-invariant probability measure on $M$ such that $\mathsf{Q}_{\#} \Lambda(\nu) = \nu$. Moreover we have that
    \begin{itemize}
        \item[i)] $\Lambda (\Prob^{ac}(M/\mathsf G)) = \Prob^{ac}(M) \cap \Prob^{\mathsf G}(M)$ and  $\Lambda(\Prob_c(M/ \mathsf G)) = \Prob_c(M) \cap \Prob^{\mathsf G}(M)$ ,
        \item[ii)] $\ell_p^M(\Lambda(\mu), \Lambda(\nu)) = \ell_p^{M/\mathsf G}(\mu, \nu)$ for any $\mu, \nu \in \Prob(M/\mathsf G)$ and if $(\mu, \nu)$ is $q$-separated then $(\Lambda(\mu), \Lambda(\nu))$ is $q$-separated,
        \item[iii)] if $(\mu_t)$ is a $q$-geodesic in  $\Prob(M/\mathsf G)$  then $(\Lambda(\mu_t) )$ is a $\mathsf G$-invariant $q$-geodesic in $\Prob(M)$.
    \end{itemize}
\end{theorem}

\begin{proof}
The first claim of the proof follows analogously to \cite[Theorem 3.2]{QuotientMondino}. We are left to show the points $ii)$ and $iii)$. Using Borel selector arguments, one can construct a Borel map 
\[
\mathsf{Sel}: (M/\mathsf G)^2 _{\leq^{\ast}}  \,  \longrightarrow \, M^2_{\leq}
\]
such that if $[x] \leq^{\ast} [y]$ then $\mathsf{Sel}([x], [y]) : = (x',y')$ where $x'\leq y'$, $\mathsf Q(x')=[x]$,  $\mathsf Q(y')=[y]$ and $\ell_M(x',y') = \ell_{M/\mathsf G}([x], [y])$. Moreover, for $(x',y') \in M^2_{\leq}$ we define the map
\[
\tau^{(x',y')} : \mathsf G \longrightarrow M^2_{\leq} \, , \, \qquad \tau^{(x',y')}(\delta) = (\tau_{\delta}(x'), \tau_{\delta}(y')) \, .
\]
Let $\mu, \nu \in \Prob(M/\mathsf G)$. Suppose that $\pi \in \Pi_{\leq}(\Lambda (\mu), \Lambda(\nu))$, then it is easy to see that $(\mathsf Q, \mathsf Q)_{\#}\pi \in \Pi_{\leq^{\ast}}(\mu, \nu)$ and moreover
\[
\ell_q^{M / \mathsf G}(\mu, \nu)^q \geq \int \ell_{M/\mathsf G} ([x], [y])^q
 \,  (\mathsf Q, \mathsf Q)_{\#}\pi (\d [x], \d [y])  \geq  \int \ell(x,y)^q \, \pi (\d x, \d y) .
\]
Taking the supremum over $\pi$, one obtains one inequality in $ii)$. On the other side, suppose that there is some $\pi \in \Pi_{\leq^{\ast}} (\mu, \nu)$, then we define
\begin{equation} \label{eq: pi hat}
    \hat \pi : = \int \left ( \tau^{\mathsf{Sel}([x],[y])} \right )_{\#} \mathfrak{h} \, \pi( \d [x], \d [y] ) \, .
\end{equation}
It is easy to check that $\hat \pi$ is well defined and $\hat{\pi} \in \Pi_{\leq} (\Lambda(\mu), \Lambda (\nu))$. In particular 
\[
\ell_q^{M}( \Lambda(\mu), \Lambda(\nu) )^q \geq \int \ell (x,y)^q \, \hat \pi (\d x, \d y) = \int \ell_{M / \mathsf G} ([x], [y])^q \, \pi (\d [x], \d [y]) \, ,
\]
hence, the other inequality is verified. Suppose that $(\mu, \nu)$ is $q$-separated by $(\pi, u, v)$. Then, define $\hat u: M \rightarrow \R \cup \{+ \infty\}$ by $\hat u(x) = u([x])$ and analogously one can define $\hat{v}: M \rightarrow \R \cup \{- \infty\}$ and $\hat \pi$ as in \eqref{eq: pi hat}. Clearly, $(\Lambda(\mu), \Lambda(\nu))$ is $q$-separated by $(\hat \pi, \hat u, \hat v)$. Finally, the last claim in $iii)$ easily follows by $ii)$.
\end{proof}

The previous theorem yields the following two corollaries.

\begin{corollary}
    Let $(M,g)$ be a globally hyperbolic spacetime and $\mathsf G$ be a compact Lie group acting by $g$-isometries on $M$with acausal orbits. Let $q \in (0,1)$, then $(M/ \mathsf G, \ell_{M/\mathsf G}, \mathsf Q _{\#} \Vol_g)$ is $q$-essentially non-branching.
\end{corollary}

\begin{proof}
First notice that a family of causal branching geodesics in $M / \mathsf G$ lift to a family of causal branching geodesics in $M$ by the Lorentzian submetry $\mathsf Q$ since $M$ is globally hyperbolic. \\
Suppose by contradiction that there exists $\boldsymbol{\pi}$ representing a $q$-geodesic between measures in $\Prob^{ac}_c(M/ \mathsf G)$ that is not concentrated on a non-branching set. Using $iii)$ of Theorem \ref{thm: analogo mondino}, one can construct a $\mathsf G$-invariant lift and by \cite[Theorem 2.43]{Octet} the corresponding dynamical plan that is concentrated on a non-branching set. This is a contradiction since $(M, g, \Vol_g)$ is $q$-essentially non-branching.
\end{proof}

The proof of the following result is inspired by \cite[Corollary 3.6]{QuotientMondino}.

\begin{corollary}
    The map $\mathsf Q_{\#} : \Prob_c(M) \rightarrow \Prob_c(M/\mathsf G)$ is a Lorentzian submetry.
\end{corollary}

\begin{proof}
Let $\overline{\mu} \in \Prob_c(M)$ and $r \geq 0$. Let $\nu \in \Prob_c(M)$ such that $\overline{\mu} \preceq \overline{\nu}$ and $\ell_q(\overline{\mu}, \overline{\nu}) \geq r$. Arguing like in the proof of Theorem \ref{thm: analogo mondino} one gets that
\[
\ell_q^{M/ \mathsf G} ( \mathsf Q_{\#} \overline{\mu}, \mathsf Q_{\#} \overline{\nu} ) \geq \ell_q (\overline{\mu}, \overline{\nu} )= r \, .
\]
On the other side, let $\nu \in \Prob_c(M / \mathsf G)$ such that $\ell_q^{M / \mathsf G} (\mathsf Q_{\#} \overline{\mu} , \nu) \geq r$ and $\mathsf Q_{\#} \overline{\mu} \preceq^{\ast} \nu$. Denote by $\mu_{\mathsf G}: = \Lambda ( \mathsf Q_{\#} \overline \mu) \in \Prob_c(M)$, then by $ii)$ of Theorem \ref{thm: analogo mondino}, we have that 
\[
\ell_q^M( \mu_{\mathsf G}, \Lambda(\nu)) = \ell_q^M( \Lambda ( \mathsf Q_{\#} \overline \mu), \Lambda(\nu)) = \ell_q^{M / \mathsf G} (\mathsf Q_{\#} \overline{\mu} , \nu) \geq r \, .
\]
Suppose for the moment that $\mu$ is absolutely continuous with respect to $\mu_{\mathsf G}$. Let $\pi \in \Pi_{\leq ^{\ast}} (\mathsf Q_{\#} \overline{\mu}, \nu)$ be an $\ell_q^{M / \mathsf G}$-optimal plan and consider the $\ell_q^M$-optimal plan $\hat \pi \in \Pi_{\leq} (\mu_{\mathsf G}, \Lambda(\nu))$ as in \eqref{eq: pi hat}. Disintegrate $\hat \pi$ with respect to the projection on the first factor $\mathsf{Pr}_1$ obtaining $\hat \pi = \int \hat \pi_x \, \mu_{\mathsf G}(\d x)$, define $\tilde \pi= \int \hat \pi_x \, \overline{\mu} (\d x)$ and pick $\overline{\nu} = (\mathsf {Pr}_2 )_{\#} \tilde \pi$.  Similarly to the proof of  \cite[Corollary 3.6]{QuotientMondino}, one can check that $\mathsf Q_{\#} \overline{\nu} = \nu$ and 
\[
\ell_q ^{M}(\mu_{\mathsf G}, \Lambda (\nu))= \ell_q^{M} (\overline{\mu}, \overline{\nu} ) \, ,
\]
concluding the proof in this case. 

The general case follow by approximation. Indeed, one can construct a sequence $\overline\mu_n \preceq \overline \mu$ such that $\overline \mu_n \in \Prob_c(M)$, $\overline \mu_n$ is absolutely continuous with respect to its $\mathsf G$-average and that $\overline \mu_n$ weakly converges to $\overline\mu$. In particular, one defines $\tilde \pi_n$ as before and so $\overline{\nu}_n := ( \mathsf{Pr}_2)_{\#}  \tilde \pi_n$ weakly converges up to subsequences to some $\overline{\nu} \preceq \overline\mu$.  By the upper semicontinuity of $\ell_q$ with respect to weak convergence, see \cite[Lemma 2.19]{Octet}, we get that $\ell^M_q(\overline{\mu}, \overline{\nu}) \geq r$.

We have showed that the map $\mathsf Q_{\#}$ is a future directed Lorentzian submetry. Analogously, one can show that $\mathsf Q_{\#}$ is also a past directed Lorentzian submetry, concluding the proof.
\end{proof}

The following theorem is a generalization in the Lorentzian setting of \cite[Theorem 3.7]{QuotientMondino}.

\begin{theorem} \label{thm: passage to quotient}
    Let $q \in (0,1)$, $K \in R$ and $N \geq 1$. Let $(M,g)$ be a globally hyperbolic spacetime such that $\dim(M) \leq N$ and $\Ric(v,v) \geq Kg(v,v)$ holds for any causal vector $v$ and let $\mathsf G$ be a compact Lie group acting by $g$-isometries on $M$ with acausal orbits. Then the quotient space $(M/\mathsf G, \ell_{M/\mathsf G}, \mathsf Q_{\#} \Vol_g)$ satisfies the $\mathsf{TCD}_q(K,N)$ condition.
\end{theorem}

\begin{proof}
    Let $(\overline{\mu}_0, \overline{\mu}_1) \in \Prob_c^{ac}(M / \mathsf G, \mathsf Q_{\#}\Vol_g)^2$ be $q$-separated. Let $(\mu_t)$ be a $q$-geodesic connecting $\overline{\mu}_0$ to $\overline{\mu}_1$ and $\pi \in \Pi_{\leq}(\overline{\mu}_0, \overline{\mu}_1)$ be an $\ell_q$-optimal coupling. 
    
    Using Theorem \ref{thm: analogo mondino} we obtain a $\mathsf G$-invariant $q$-geodesic $\Lambda(\mu_t)= \hat \mu_t$ that connects the $q$-separated measures $( \Lambda (\overline \mu_0), \Lambda (\overline{\mu}_1)) \in \Prob^{ac}_c(M)$ and a lifted optimal plan $\hat \pi$ as in \eqref{eq: pi hat}. 
    
    Necessarily, we have that $\hat \mu_t = \rho_t \Vol_g$ and, by Theorem \ref{thm: compatibility} we get
    \[
    -\int \hat \rho_t(x)^{1- \frac{1}{N}} \, \Vol_g(\d x) \leq - \int \tau_{K, N'}^{(1-t)} (\ell(x,y)) \,  \hat \rho_0(x)^{-\tfrac{1}{N'}} + \tau_{K, N'}^{(t)} (\ell(x,y)) \,  \hat \rho_1(y)^{-\tfrac{1}{N'}} \, \hat \pi (\d x,\d y) 
    \]
    for every $t \in [0,1]$ and $N' \geq N$. The density $\hat \rho_t$ is $\mathsf G$-invariant and so we can define $\rho_t ([x]) = \hat \rho_t (x)$. In particular, $\mu_t = \mathsf Q_{\#} \hat \mu_t = \rho_t \, \mathsf Q_{\#} \Vol_g$. 
    
Therefore, the change of variable formula implies that \eqref{eq: TCD ineq} holds and so, the conclusion follows.
\end{proof}

\begin{ack}
The author acknowledges the support of the European Union - NextGenerationEU, in the framework of the PRIN Project `Contemporary perspectives on geometry and gravity' (code 2022JJ8KER – CUP G53D23001810006). The views and opinions expressed are solely those of the authors and do not necessarily reflect those of the European Union, nor can the European Union be held responsible for them. 
    
The author is grateful for the support of the Fields Institute for Research in Mathematical Sciences (Toronto, Canada) during the Thematic Program on Shocks and Singularities: Nonlinear evolution equations in physical and life sciences.

Finally, the author wants to thank Nicola Gigli for interesting discussions and comments.
\end{ack}

\bibliography{bibliography} 
\bibliographystyle{abbrv}

\end{document}